\documentclass[12pt, reqno]{amsart}

\usepackage{hyperref}
\usepackage{amssymb, amsthm, amsmath, amsfonts,epsfig,bbm}
\usepackage[numbers,square]{natbib}
\usepackage{hhline}
\usepackage[left=2.3cm, top=2.3cm,bottom=2.3cm,right=2.3cm]{geometry}

\usepackage[numbers,square]{natbib}
\usepackage{array}
\usepackage{makecell,tabularx}
\usepackage[british]{babel}
\usepackage{amsfonts}              

\makeatletter
\newcommand{\addresseshere}{%
  \enddoc@text\let\enddoc@text\relax
}
\makeatother

\numberwithin{equation}{section}

\newcommand{\Mod}[1]{\ (\mathrm{mod}\ #1)}

\newcommand{\stirling}[2]{\genfrac{[}{]}{0pt}{}{#1}{#2}}
\newcommand{\stirlingb}[2]{\genfrac{\{}{\}}{0pt}{}{#1}{#2}}

\newtheorem {theorem}{Theorem}[section]
\newtheorem {proposition}[theorem]{Proposition}
\newtheorem {lemma}[theorem]{Lemma}

{\theoremstyle{definition}

}

{
\theoremstyle{remark}
\newtheorem {remark}[theorem]{Remark}
}

\newcommand{\Z}{\mathbb{Z}}

\DeclareMathOperator*{\cotanh}{cotanh}

\def\ba{\begin{array}}
\def\ea{\end{array}}
\def\bea{\begin{eqnarray} \label}
\def\eea{\end{eqnarray}}
\def\be{\begin{equation} \label}
\def\ee{\end{equation}}
\def\bit{\begin{itemize}}
\def\eit{\end{itemize}}
\def\ben{\begin{enumerate}}
\def\een{\end{enumerate}}

\def\E{\mathbb{E}}

\def\II{\mathbb{I}}
\def\JJ{\mathbb{J}}

\def\N{\mathbb{N}}
\def\P{\mathbb{P}}

\def\R{\mathbb{R}}
\def\RRd1{\mathbb{R}^{d+1}}

\def\bS{\mathbb{S}}

\def\b{\beta}

\def\n{\eta}

\def\r{\varrho}

\def\G{\Gamma}

\def\dint{\textup{d}}

\newcommand{\eee}{{\rm e}}

\newcommand{\ind}{\mathbbm{1}}
\newcommand{\eps}{\varepsilon}
\newcommand{\pos}{\mathop{\mathrm{pos}}\nolimits}

\newcommand{\conv}{\mathop{\mathrm{conv}}\nolimits}

\newcommand{\dd}{{\rm d}}

\begin{document}

\title[Expected $f$-vector of the Poisson Zero Polytope]{Expected $f$-vector of the Poisson Zero Polytope and Random Convex Hulls in the Half-Sphere}

\author{Zakhar Kabluchko}
\address{Zakhar Kabluchko: Institut f\"ur Mathematische Stochastik,
Westf\"alische Wilhelms-Universit\"at M\"unster,
Orl\'eans-Ring 10,
48149 M\"unster, Germany}
\email{zakhar.kabluchko@uni-muenster.de}

\date{}

\begin{abstract}
We prove an explicit combinatorial formula for the expected number of faces of the zero polytope of the homogeneous and isotropic Poisson hyperplane tessellation in $\R^d$. The expected $f$-vector is expressed through the coefficients of  the polynomial
$$
(1+ (d-1)^2x^2) (1+(d-3)^2 x^2) (1+(d-5)^2 x^2) \ldots.
$$
Also, we compute explicitly the expected $f$-vector and the expected volume of the spherical convex hull of $n$  random points  sampled uniformly and independently from the $d$-dimensional half-sphere. In the case when $n=d+2$, we compute the probability that this spherical convex hull is a spherical simplex, thus solving the half-sphere analogue of  the Sylvester four-point problem.
\end{abstract}


\maketitle

\section{Main results}\label{sec:results}

\subsection{Poisson zero polytope}
Poisson hyperplane processes and the corresponding random tessellations of the Euclidean space by polytopes have been extensively studied in stochastic geometry since the works of Miles~\cite{miles_thesis,miles_polygons,miles_flats1,miles_synopsis,miles_flats2,miles_aggregates} and Matheron~\cite{matheron72,matheron74,matheron_book}; see Section~4.4 and Chapter~10 of the book by Schneider and Weil~\cite{SW08} for more information and references,  as well as~\cite{goudsmith,richards,santalo,santalo_yanez} for early contributions.
Stationary Poisson hyperplane tessellations give rise to (at least) two natural random polytopes: \textit{the Poisson zero polytope} (defined as the a.s.\ unique polytope of the tessellation containing the origin) and the \textit{typical Poisson polytope} (defined essentially as a polytope picked uniformly at random from the set of polytopes of the tessellation contained in some very large observation window). One of the most interesting characteristics of a random polytope is its expected $f$-vector whose $k$-th component is, by definition, the expected number of $k$-dimensional faces of the polytope. While it is well known that the expected $f$-vector of the typical Poisson polytope coincides  with the $f$-vector of the cube of the same dimension (see, for example, Theorems 10.3.1 and 10.3.2 in~\cite{SW08}), an exact result for the zero polytope is missing.
In the present paper we close this gap by providing an explicit formula for the expected $f$-vector of the zero polytope of the isotropic and homogeneous Poisson hyperplane tessellation on $\R^d$.

Let us recall the definitions of the Poisson hyperplane process and the Poisson zero polytope. Denote by $\|\cdot\|$ and $\langle \cdot, \cdot \rangle$ the Euclidean norm and the standard scalar product on $\R^d$, respectively.  Let $\bS^{d-1}= \{x\in\R^d\colon \|x\| = 1\}$ be the unit sphere in $\R^d$. Let $A(d,d-1)$ be the Grassmannian manifold of all affine hyperplanes in $\R^d$. Every affine hyperplane $H\in A(d,d-1)$ can be represented in the form
$$
H = H(w,\tau) := \{x\in\R^d\colon \langle x,w\rangle = \tau \}
$$
with some ``direction'' $w\in \bS^{d-1}$ and some (possibly negative) ``distance'' $\tau \in\R$.  In fact, $H(w,\tau) = H(-w,-\tau)$, and every hyperplane not passing through the origin has exactly two such representations.
A homogeneous \textit{Poisson hyperplane process} with intensity $\gamma>0$ is a random, countable collection of affine hyperplanes $X=\{H(W_i,T_i)\}_{i\in\Z}$, where
\begin{itemize}
\item [(a)] $\{T_i\}_{i\in\Z}$ are the arrivals of a homogeneous, intensity $\gamma$ Poisson process on the real line;
\item[(b)] $\{W_i\}_{i\in\Z}$ are independent, identically distributed  random vectors with certain centrally symmetric probability distribution $\mu$ on the unit sphere $\bS^{d-1}$;
\item[(c)] $\{T_i\}_{i\in\Z}$ is independent of $\{W_i\}_{i\in\Z}$.
\end{itemize}
Equivalently, we can view the Poisson hyperplane process $X$ as  a Poisson point process on $A(d,d-1)$ whose intensity measure $\Theta$ is given by
$$
\Theta(A) := \gamma \int_{\bS^{d-1}} \left(\int_{-\infty}^{+\infty} \ind_{\{H(w,\tau) \in A\}} \dd \tau \right)\mu (\dd w),
$$
for all Borel sets $A\subset A(d,d-1)$; see~\cite[Section~4.4]{SW08}. In the present paper, we restrict our attention to the \textit{isotropic} case meaning that the direction measure $\mu$ is chosen to be the uniform probability distribution on $\bS^{d-1}$. Without restriction of generality, we may choose $\gamma:=1$.  It is known that $X$ consists of countably many random affine hyperplanes whose probability law is invariant with respect to the natural action of the isometry group of $\R^d$ on the set of hyperplanes $A(d,d-1)$.  The hyperplanes of the Poisson hyperplane process $X$ dissect $\R^d$ into countably many polytopes; see the left panel of Figure~\ref{fig:poisson} for a realization when $d=2$.  The \textit{Poisson zero polytope} or the \textit{Crofton polytope} is the a.s.\ unique polytope of this tessellation that contains the origin.

\begin{figure}[t]
\begin{center}
\includegraphics[width=0.49\textwidth ]{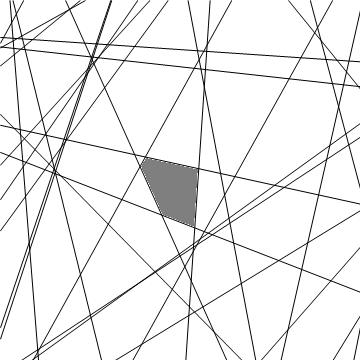}
\includegraphics[width=0.49\textwidth]{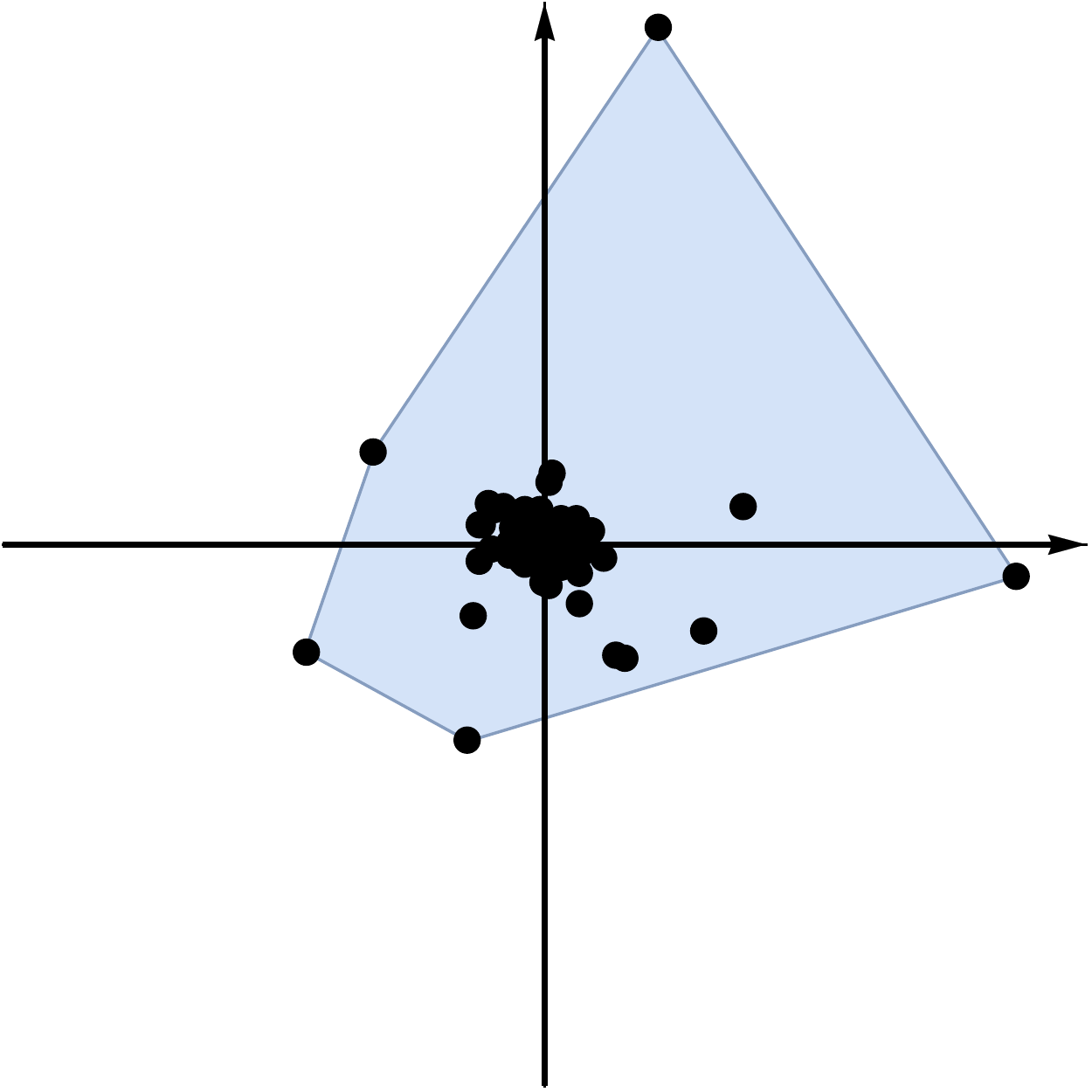}
\end{center}
\caption{
Left: The Poisson line tessellation in the plane, together with the zero polygone.
Right: The dual Poisson point process $\Pi_{2,1}$ on $\R^2$ with intensity $\|x\|^{-3}$, together with its convex hull. The lines of the tessellation correspond to the points of the Poisson process via  projective duality.
}
\label{fig:poisson}
\end{figure}

\subsection{Statement of the main result}\label{subsec:statement_main}
The number of $k$-dimensional faces of a $d$-dimensional polytope $P$ is denoted by $f_k(P)$, for $k\in \{0,1,\ldots,d\}$. For example, $f_0(P)$ is the number of vertices, $f_1(P)$ is the number of edges, $f_{d-1}(P)$ is the number of $(d-1)$-dimensional faces (called \textit{facets}), and $f_d(P)=1$. The \textit{$f$-vector} of the polytope $P$ is defined by $\mathbf f(P):= (f_0(P),\ldots, f_{d-1}(P))$.
The main result of the present paper is the following formula for the expected $f$-vector of the Poisson zero polytope.

\begin{theorem}\label{theo:main}
Let $\mathcal Z_d$ be the $d$-dimensional Poisson zero polytope with $d\in\N$. Then, for all $\ell\in \{0,\ldots,d\}$ we have
\begin{equation}\label{eq:theo:main}
\E f_\ell(\mathcal Z_d) = \frac{\pi^{d-\ell}}{(d-\ell)!} A[d,d-l],
\end{equation}
where $A[n,k]$, $n\in\N_0$, $k\in \Z$, is an array of numbers defined as follows. Consider a sequence of polynomials given by $Q_0(x) = Q_1(x) = 1$ and, for $n\in\{2,3,\ldots\}$,
\begin{align}
Q_n(x)
&= \prod_{\substack{j\in \{1,\ldots,n-1\} \\ j \not\equiv n \Mod{2}}} (1+ j^2 x^2) \label{eq:def_Q_n}\\
&=
\begin{cases}
(1+ (n-1)^2x^2) (1+(n-3)^2 x^2)\ldots (1 + 3^2x^2) (1 + 1^2 x^2),&\text{ if $n$ is even,}\\
(1+ (n-1)^2x^2) (1+(n-3)^2 x^2)\ldots (1+ 4^2 x^2) (1+ 2^2 x^2),&\text{ if $n$ is odd}.
\end{cases}
\notag
\end{align}
Then,
\begin{equation}\label{eq:def_A_n_k}
A[n,k] =
\begin{cases}
[x^k]  Q_n(x), &\text{ if $k$ is even},\\
[x^k] \left(\tanh \left(\frac{\pi}{2x}\right)  \cdot  Q_n(x)\right), &\text{ if $k$ is odd and $n$ is even},\\
[x^k] \left(\cotanh \left(\frac{\pi}{2x}\right) \cdot Q_n(x)\right), &\text{ if $k$ is odd and $n$ is odd}.
\end{cases}
\end{equation}
Here, $[x^k] H(x)$ denotes the coefficient of $x^k$ in $H(x)$, a formal power series in positive and negative powers of $x$.
\end{theorem}


Let us list some properties of the numbers $A[n,k]$ that follow from~\eqref{eq:def_A_n_k} and~\eqref{eq:def_Q_n}. First of all,  $A[n,k]=0$ if $k>n$, so let $k\leq n$ in the following.  By definition, $A[n,k]$ is an integer number if $k$ is even. These values are ``nice'', whereas the values with odd $k$ are ``ugly''.  Specifically,  in the case when $k\geq 1$ is odd, $A[n,k]$ is a polynomial of $\pi^2$ with rational coefficients, if $n$ is even, and $1/\pi$ times such a polynomial, if $n$ is odd.
For $k=0$ we trivially have $A[n,0] = 1$ for all $n\in\N_0$. Even though~\eqref{eq:theo:main} involves $A[n,k]$'s with $k\geq 0$ only, the negative values of $k$ are also of interest since $A[n,-1]$ will appear in Theorems~\ref{theo:spherical_polytope_f_vector}, \ref{theo:expected_angle} and~\ref{theo:sylvester}, below.
If $k\leq -1$ is odd, then $A[n,k]$ is a polynomial in odd powers of $\pi$. If $k\leq -2$ is even, then $A[n,k]=0$.
As a consequence of these observations and~\eqref{eq:theo:main},  the expected face numbers $\E f_\ell(\mathcal Z_d)$ are rational multiples of powers of $\pi^2$ if the codimension $d-\ell$ is even (the ``nice'' case),  or polynomials in $\pi^2$ with rational coefficients if the codimension $d-\ell$ is odd (the ``ugly'' case).
Explicit values of the expected $f$-vectors of $\mathcal Z_d$ in dimensions up to $d=10$ and the values of $A[n,k]$ can be found in
Tables~\ref{tab:f_vect_Poisson}, \ref{tab:A_n_k_0} and~\ref{tab:A_n_k} at the end of the present paper. 




As it turns out, the ``ugly'' values of $\E f_\ell(\mathcal Z_d)$ are determined uniquely by the ``nice'' values  and the Dehn-Sommerville relations which we are now going to recall.
The random polytope $\mathcal Z_d$ is \textit{simple} with probability $1$ (that is, each vertex of this polytope is adjacent to exactly $d$ edges (and also exactly $d$ facets). Equivalently,  the dual polytope of $\mathcal Z_d$ (which will be described explicitly in Section~\ref{subsec:convex_half_sphere}) is \textit{simplicial} with probability $1$, that is all of its facets (and, consequently, all faces) are simplices; see~\cite[Section~4.5]{GruenbaumBook} for a discussion of these classes of polytopes.  For a simple $d$-dimensional polytope $P$, the \textit{Dehn-Sommerville relations} (see, e.g.,\ \cite[Section~9.2]{GruenbaumBook}) state that
\begin{equation}\label{eq:dehn_sommerville_deterministic}
f_\ell(P) =  \sum_{i=0}^{\ell} (-1)^i \binom{d-i}{d-\ell} f_i(P),
\end{equation}
for all  $\ell\in \{0,\ldots,d\}$.
Applying them to $P:=\mathcal Z_d$ and taking the expectation, we arrive at the equations
\begin{equation}\label{eq:dehn_sommerville_Ef}
\E f_\ell(\mathcal Z_d) = \sum_{i=0}^{\ell} (-1)^i \binom{d-i}{d-\ell} \E f_i(\mathcal Z_d),
\end{equation}
for all $\ell\in \{0,\ldots,d\}$. As we shall explain below, these equations allow to express the ``ugly'' values of $\E f_\ell(\mathcal Z_d)$ with odd $d-\ell$ as linear combinations of the ``nice'' values with even $d-\ell$. The coefficients are essentially the Taylor coefficients of $\tanh$ and $\cotanh$, which can be expressed through Bernoulli numbers~\cite[\S 6.5]{graham_knuth_patashnik_book}.  The main difficulty is therefore to prove Theorem~\ref{theo:main} in the case when $d-\ell$ is even.



\subsection{Known results}\label{sec:related_results}
Previously, the values of $\E f_\ell (\mathcal Z_d)$ have been determined only in some special cases by the methods of stochastic and integral geometry, whereas their combinatorial structure described in Theorem~\ref{theo:main} remained unknown.
Specifically, it has been known, see~\cite{schneider_anisotropic} or~\cite[Theorem~10.4.9]{SW08},  that
\begin{equation}\label{eq:E_f_0_Poisson_Poly}
\E f_0(\mathcal Z_d) = \frac{d!}{2^d} \kappa_d^2 = \frac{\kappa_d}{\kappa_{d-1}}\pi^{d-1},
\end{equation}
where $\kappa_d:= \pi^{d/2}/\Gamma(\frac d2+1)$ is the volume of the $d$-dimensional unit ball, and the equality of both expressions on the right-hand side follows from the Legendre duplication formula.
For $d=2$, this formula, which takes the form $\E f_0(\mathcal Z_2) = \pi^2/2$, can be found in the work of R\'enyi and Sulanke~\cite{renyi_sulanke_ringgebiet}. For $d>2$, the formula has been derived by Sulanke and Wintgen~\cite{sulanke_wintgen} with a contribution by Schmidt~\cite{schmidt_some_results}. Finally, a formula for $\E f_0(\mathcal Z_d)$ in the case of a not necessarily isotropic Poisson hyperplane tessellation has been obtained by Schneider~\cite{schneider_anisotropic} who also proved that the maximum of the expected number of vertices is attained if the direction distribution of the hyperplanes is isotropic. Actually, these papers considered random polytopes defined as intersections of  finitely many random half-spaces, and the Poisson zero cell appears as the limit of this model as the number of half-spaces goes to infinity.

Equation~\eqref{eq:E_f_0_Poisson_Poly} also yields a formula for $\E f_1(\mathcal Z_d)$ via the a.s.\ relation $2 f_1(\mathcal Z_d) = d f_0(\mathcal Z_d)$ that is valid for every simple polytope. Finally, it has been known that
\begin{equation}\label{eq:E_f_d-2_Poisson_Poly}
\E f_{d-2} (\mathcal Z_d) = \frac{1}{2} \binom {d+1}{3} \pi^2,
\end{equation}
see~\cite[Equation~(1.16)]{beta_polytopes}, where it is explained how this can be derived from a result of~\cite{barany_etal}. All these results are consistent with Theorem~\ref{theo:main}.
It seems that no formulae for $\E f_k(\mathcal Z_d)$ have been known for $k\notin \{0,1,d-2,d\}$.  Asymptotic properties of the expected $f$-vector of the Poisson zero polytope (and some more general random polytopes), as $d\to\infty$, have been studied in~\cite{HoermannHugReitznerThaele}.  Some refinements of these results were obtained in~\cite[Section~1.7]{beta_polytopes}, and it should be possible to obtain even more refined results using the exact formula stated in Theorem~\ref{theo:main}.

Let us also mention that the expected intrinsic volumes of the Poisson zero polytope can be expressed through its expected face numbers~\cite[p.~693]{SchneiderWeightedFaces}. Thus, Theorem~\ref{theo:main} also yields an explicit formula for the expected intrinsic volumes of $\mathcal Z_d$.
In fact,  for all $d\in \N$ and $\ell\in\{0,\ldots,d\}$, the expected $\ell$-th intrinsic volume of the zero cell of the stationary and isotropic Poisson hyperplane tessellation in $\R^d$ with intensity $\gamma>0$ is given by
\begin{align*}
\E V_\ell(\mathcal Z_d)
= \left(\frac {2\pi}{\gamma}\right)^\ell
\left(\frac {\Gamma(\frac{d+1}2)}{\Gamma(\frac d2)}\right)^\ell
\frac{\Gamma(\frac \ell 2+1)}{\ell!}
A[d,\ell].
\end{align*}
The details of the calculation leading to this formula will be presented in~\cite{kabluchko_thaele_spherical_tess}, where also some related  facts (for example, a similar statement for the north pole cell of a spherical hyperplane tessellation) will be derived.

\subsection{Recurrence relations}\label{subsec:recurrence_relations}
It is easy to check that the numbers $A[n,k]$ satisfy the recurrence relation
\begin{equation}\label{eq:rel_A_n_k_intro}
A[n+2,k] - A[n,k] = (n+1)^2 A[n,k-2],
\end{equation}
see Lemma~\ref{lem:rec_first_kind}, below. In the proof of Theorem~\ref{theo:main}, an important role will be played by the numbers
\begin{equation}\label{eq:def_B_n_k_intro}
B\{n,k\} := \frac {1}{(k-1)!(n-k)!} \int_0^\pi (\sin x)^{k-1} x^{n-k} \dint x, \quad
k\in\{1,\ldots,n\},
\end{equation}
which will be shown to  satisfy the ``dual'' relation
\begin{equation}\label{eq:rel_B_n_k_intro}
B\{n,k-2\} - B\{n,k\} = (k-1)^2 B\{n+2,k\},
\end{equation}
see Lemma~\ref{lem:rec_second_kind}, below.
Formally, these relations transform into each other under the substitution $(n,k) \mapsto (-k,-n)$.
These properties bear some similarity to the well-known properties~\cite[\S 6.1]{graham_knuth_patashnik_book} of the Stirling numbers $\stirling{n}{k}$ and $\stirlingb{n}{k}$:
\begin{equation}\label{eq:stirling_relations}
\stirling{n+1}{k} - \stirling{n}{k-1} = n\stirling{n}{k},
\qquad
\stirlingb{n+1}{k} - \stirlingb{n}{k-1} = k\stirlingb{n}{k}.
\end{equation}
For even $k$ there is also certain similarity between the formula for $A[n,k]$  stated in~\eqref{eq:def_A_n_k} and~\eqref{eq:def_Q_n} and the following definition of the Stirling numbers of the first kind:
$$
\stirling{n}{k} = [x^k]  \left(x(x+1)\ldots(x+n-1)\right), \qquad n\in \N,\;\; k\in \{0,\ldots,n\}.
$$
It is also known that after extending the definition of the  Stirling numbers to negative arguments, the relation
\begin{equation}\label{eq:stirling_reciprocity}
\stirlingb{n}{k} = \stirling{-k}{-n}
\end{equation}
holds. It turns out that after an appropriate analytic continuation, a certain analogue of~\eqref{eq:stirling_reciprocity} holds for $A[n,k]$ and $B\{n,k\}$.  Since these issues are not directly related to the proof of Theorem~\ref{theo:main}, we shall discuss them elsewhere and in a more general context.  The analogy between the Stirling numbers and the arrays $A[n,k]$ and $B\{n,k\}$ explains  our notation with square and curly brackets.



It is natural to ask whether the array $A[n,k]$ is uniquely determined by the recurrence relations~\eqref{eq:rel_A_n_k_intro} together with some natural boundary conditions. It turns out that in order to enforce uniqueness it is necessary to add a certain Euler-type relation. This is stated in the following proposition.
\begin{proposition}\label{prop_A_n_k_algorithm}
The triangular array $A[n,k]$, where $n\in\N_0$ and $k\in \{0,\ldots,n\}$,  is uniquely determined by the following properties:
\begin{itemize}
\item[(i)] $A[n,0]=1$ for all $n\in\N_0$;
\item[(ii)] $A[n,n]= 2^{-n} (n!)^2 / \Gamma(\frac n2 +1)^2$ and $A[n,n-1] = \frac \pi 2 A[n,n]$ for all $n\in \N$;
\item[(iii)] $A[n,k] = A[n-2,k] + (n-1)^2 A[n-2,k-2]$ for all $n\geq 4$ and $k\in \{2,\ldots,n-2\}$;
\item[(iv)] $A[n,1] = \frac 1\pi((-1)^{n-1} + 1 + \sum_{k=2}^n (-1)^k (\pi^k/ k!) A[n,k])$ for all $n\geq 3$.
\end{itemize}
\end{proposition}
\begin{proof}
Part (ii) is equivalent to the formula  $\E f_0(\mathcal Z_n) = 2^{-n} n! \pi^{n}/\Gamma(\frac n2+1)^2$ together with the relation $2 f_1(\mathcal Z_n) = n f_0(\mathcal Z_n)$. Both results were already mentioned in Section~\ref{sec:related_results}, with $n$ replaced by $d$.  Part (iii) is just a restatement of~\eqref{eq:rel_A_n_k_intro}, whereas Part (iv) follows from the Euler relation $\sum_{k=0}^n (-1)^{k} \E f_k(\mathcal Z_n) = 1$ and~\eqref{eq:theo:main}. The fact that (i)-(iv) determine the $A[n,k]$'s \textit{uniquely} easily follows by induction over $n$. Indeed, (i) and (ii) determine $A[0,k]$, $A[1,k]$ and $A[2,k]$ for all admissible $k$'s, which is the base of induction. Assuming that the $A[m,k]$'s are determined uniquely for all $m\in \{0,\ldots,n-1\}$,  $k\in \{0,\ldots,m\}$ with some $n\geq 3$,  we can use (i), (ii) and (iii) to determine $A[n,k]$ for all $k \in \{0,\ldots,n\}\backslash\{1\}$. Finally, (iv) determines $A[n,1]$, thus completing the induction. Note that the Euler-type relation (iv) cannot be removed without loosing uniqueness since the value $A[3,1]$ is not determined uniquely by the remaining conditions.
\end{proof}

Let us finally mention that the triangular array $A[n,2k]$ is the row-reverse of Entry~A121408 and the unsigned version of Entry~A182971 in~\cite{sloane}.  The numbers $A[n,2k]$ can be expressed through the central factorial numbers defined as the coefficients in the expansion of $x^{[n]}:= x (x+\frac n2 -1) \ldots (x-\frac n2 +1)$.

\section{Convex hulls on the half-sphere}\label{sec:convex_half_sphere}
\subsection{Description of the model}\label{subsec:convex_half_sphere}
We are now going to state some applications of Theorem~\ref{theo:main} to a natural class of random spherical polytopes.
Let $U_1,\ldots,U_n$ be random points sampled uniformly and independently from the $d$-dimensional upper half-sphere
$$
\bS^d_+:=\{x = (x_0,\ldots,x_{d})\in \R^{d+1}: x_0\geq 0, \|x\|=1\}.
$$
The polyhedral convex cone  generated by these points (also known as their positive hull) is denoted by
$$
C_n = \pos(U_1,\ldots,U_n) = \left\{\sum_{i=1}^n \lambda_i U_i: \lambda_1,\ldots,\lambda_n\geq 0\right\}.
$$
Our aim is to compute expectations of various functionals of the random cone $C_n$ and the associated  random spherical polytope $C_n\cap \bS^d_+$. One example are their $f$-vectors which are related by $f_{k+1}(C_n) = f_k(C_n\cap \bS^d_+)$, for all $k\in \{0,\ldots,d\}$.  The random spherical polytope $C_n\cap \bS^d_+$ was first studied by B\'ar\'any, Hug, Reitzner and Schneider in~\cite{barany_etal}. Among other results, these authors computed the expected facet number $\E f_{d-1}(C_n\cap \bS^d_+)$, the expected surface area and spherical mean width of $C_n\cap \bS^d_+$, and showed that $\E f_0(C_n\cap \bS^d_+)$ converges to a finite limit expressed as a multiple integral.
These studies were continued in~\cite{convex_hull_sphere}, where it was shown that $C_n$ is closely related to convex hulls of certain Poisson processes.
Namely, let $\Pi_{d,1}$ be the Poisson point process on $\R^d\backslash\{0\}$ with intensity $\|x\|^{-d-1}$; see the right panel of Figure~\ref{fig:poisson} for a realization when $d=2$.  The convex hull of the atoms of this point process is denoted by $\conv \Pi_{d,1}$. Even though the number of atoms is a.s.\ infinite (because they cluster at $0$), this convex hull is a (random) polytope containing the origin in its interior, with probability $1$; see~\cite[Corollary 4.2]{convex_hull_sphere}. It is known that
\begin{equation}\label{eq:E_f_ell_Z_d_Poisson_Process}
\E f_{\ell}(\mathcal Z_d)= \E f_{d-\ell-1}(\conv \Pi_{d,1}) = \lim_{n\to\infty} \E f_{d-\ell-1}(C_n\cap \bS^d_+)
\end{equation}
for all $\ell\in \{0,1,\ldots,d-1\}$. The second equality was obtained in~\cite[Theorem 2.4]{convex_hull_sphere}.  In particular, the expected $f$-vector of $C_n\cap \bS^d_+$ converges, as $n\to\infty$, to a finite limit without any normalization, which is in sharp contrast to what is known in the setting of random convex hulls in flat convex bodies, where the $f$-vectors diverge to $\infty$.
The first equality in~\eqref{eq:E_f_ell_Z_d_Poisson_Process} follows from the observation made in~\cite[Theorem~1.23]{beta_polytopes} that $\Pi_{d,1}$ is the dual polytope of $\mathcal Z_d$, up to rescaling. In fact, the polar hyperplanes of the points of $\Pi_{d,1}$, with respect to the unit sphere,  form an isotropic and homogeneous Poisson hyperplane tessellation with intensity $\gamma= \pi^{d/2}/\Gamma(\frac d2)$; see Theorem~1.23 and Remark~1.24  in~\cite{beta_polytopes}.

\subsection{The expected \texorpdfstring{$f$-vector}{f-vector}}
First of all, we are able to identify the limit in~\eqref{eq:E_f_ell_Z_d_Poisson_Process} as follows.
\begin{theorem}\label{theo:cor}
For all $d\in\N$ and $k \in \{0,\ldots,d-1\}$,
$$
\E f_{k}(\conv \Pi_{d,1}) = \lim_{n\to\infty} \E f_{k}(C_n\cap \bS^d_+) = \frac{\pi^{k+1}}{(k+1)!} A[d,k+1].
$$
\end{theorem}
\begin{proof}
Combine~\eqref{eq:E_f_ell_Z_d_Poisson_Process} with Theorem~\ref{theo:main} and put $k :=d-\ell-1$.
\end{proof}

Previously, only the following two special cases of Theorem~\ref{theo:cor} were known with explicit limits:
$$
\lim_{n\to\infty} \E f_{d-1}(C_n\cap \bS^d_+) =  \frac{d!}{2^d} \kappa_d^2,
\qquad
\lim_{n\to\infty} \E f_{1}(C_n\cap \bS^d_+) = \frac{1}{2} \binom {d+1}{3} \pi^2.
$$
The first identity was established in~\cite[Theorem~3.1]{barany_etal}, while the second one can be found in~\cite[Remark~2.5]{convex_hull_sphere}. Via the duality between the polytopes $\mathcal Z_d$ and $\conv \Pi_{d,1}$, these identities are equivalent to the corresponding properties of the Poisson zero polytope $\mathcal Z_d$ stated in~\eqref{eq:E_f_0_Poisson_Poly} and~\eqref{eq:E_f_d-2_Poisson_Poly}.

In fact, we can even compute the complete expected $f$-vector of the spherical polytope $C_n\cap \bS^d_+$ for every finite $n$.
\begin{theorem}\label{theo:spherical_polytope_f_vector}
For all $d\in\N$, $n\geq d+1$, and all $k\in \{0,\ldots,d-1\}$, we have
\begin{equation}\label{eq:theo:spherical_polytope_f_vector}
\E f_{k} (C_n\cap \bS^d_+) = \frac{n!\pi^{k+1-n}}{(k+1)!} \sum_{\substack{s=0,1,\ldots \\ d-2s\geq k+1}} B\{n, d-2s\}(d-2s-1)^2 A[d-2s-2,k-1],
\end{equation}
where $A[n,k]$ and $B\{n,k\}$ were defined in~\eqref{eq:def_A_n_k} and~\eqref{eq:def_B_n_k_intro}, respectively.
\end{theorem}
\begin{remark}\label{rem:k=0_interpretation}
In the case $k=0$, the right-hand side of~\eqref{eq:theo:spherical_polytope_f_vector} may involve the summand $0^2 A[-1,-1]$ which should be interpreted as $A[1,1] - A[-1,1] = A[1,1] = 2/\pi$ in view of~\eqref{eq:rel_A_n_k_intro}. This convention can alternatively be explained by defining the meromorphic continuation of the function $A[x,x]$ via $A[x,x] := 2^{-x} \Gamma(x+1)^2 / \Gamma(\frac x2 +1)^2$, for $x\in\mathbb C$. Then, $x=-1$ is a pole and we have $(x+1)^2 A[x,x] \to 2/\pi$ as $x\to -1$. The same convention applies to several other results in this paper.
\end{remark}

It is interesting to compare Theorems~\ref{theo:cor} and~\ref{theo:spherical_polytope_f_vector}
with the results of Cover and Efron~\cite{cover_efron} (see also~\cite{HS15} for a recent work in this direction) who computed the expected $f$-vector of the random polyhedral cone $D_n:=\pos(V_1,\ldots,V_n)$ generated by $n$ i.i.d.\ random vectors $V_1,\ldots,V_n$ with uniform distribution on the \textit{whole} sphere $\bS^d\subset \R^{d+1}$. They gave an explicit formula for the conditional expectation $\E [f_{k}(D_n\cap \bS^d) | \{D_n \neq \R^{d+1}\}]$ in terms of binomial coefficients and proved that
$$
\lim_{n\to\infty} \E [f_{k}(D_n\cap \bS^d) | \{ D_n \neq \R^{d+1}\} ] = 2^{k+1} \binom{d}{k+1}
$$
for all $k\in \{0,\ldots, d-1\}$; see~\cite[Theorem~3']{cover_efron}. The number on the right-hand side is the number of $k$-faces of the $d$-dimensional crosspolytope (as Cover and Efron~\cite[Section~4]{cover_efron} observed in the setting of the dual cones). The event  $\{D_n \neq \R^{d+1}\}$ occurs iff there is a (random) half-space containing the vectors $V_1,\ldots,V_n$, whereas in our Theorems~\ref{theo:cor} and~\ref{theo:spherical_polytope_f_vector} we condition on the event that these vectors are in some fixed (deterministic) half-space. These very similar looking types of conditioning lead to two completely different  limits of the $f$-vector.

Let us now consider some special cases of Theorem~\ref{theo:spherical_polytope_f_vector}.  In the case $k=d-1$, \citet[Theorem~3.1]{barany_etal} showed that
\begin{equation}\label{eq:barany_etal_formula}
\E f_{d-1} (C_n\cap \bS^d_+) = \binom n d \frac {2\omega_d}{\omega_{d+1}} \int_0^\pi (\sin x)^{d-1} \left(\frac {x}{\pi}\right)^{n-d}\dd x
\end{equation}
with $\omega_{d+1} = 2\pi^{(d+1)/2}/\Gamma(\frac {d+1}2)$ being the surface measure of the $d$-dimensional unit sphere $\bS^d\subset \R^{d+1}$.
To see that this result is a special case of Theorem~\ref{theo:spherical_polytope_f_vector}, let us first consider the case when $d$ is even. Taking $k=d-1$ in Theorem~\ref{theo:spherical_polytope_f_vector}  and observing that $A[d-2,d-2]=(d-3)!!^2$, we obtain
$$
\E f_{d-1} (C_n\cap \bS^d_+)
=
\frac{n! \pi^{d-n}}{d!} (d-1)!!^2 B\{n,d\}
=
\binom n d \frac{(d-1)!!^2}{(d-1)!} \int_0^\pi (\sin x)^{d-1} \left(\frac {x}{\pi}\right)^{n-d}\dd x.
$$
This is consistent with~\eqref{eq:barany_etal_formula},  as one can see using the formulae $\Gamma(m) = (m-1)!$ and $\Gamma(m+\frac 12) = 2^{-m} (2m-1)!! \sqrt \pi$, for $m\in\N$, or the Legendre duplication formula. Let now $d\geq 3$ be odd. Then, we may take $k=d-2$ in Theorem~\ref{theo:spherical_polytope_f_vector} and use the relation $f_{d-1}(C_n\cap \bS^d_+) = \frac 2 d f_{d-2}(C_n\cap \bS^d_+)$. Bearing in mind that $A[d-2,d-3] = (d-3)!!^2$, we arrive at
\begin{align*}
\E f_{d-1} (C_n\cap \bS^d_+)
&=
\frac 2d \, \E f_{d-2} (C_n\cap \bS^d_+)
=
\frac 2d \, \frac{n! \pi^{d-1-n}}{(d-1)!} (d-1)!!^2 B\{n,d\} \\
&=
\binom n d  \frac 2\pi \frac{(d-1)!!^2}{(d-1)!} \int_0^\pi (\sin x)^{d-1} \left(\frac {x}{\pi}\right)^{n-d}\dd x,
\end{align*}
which is consistent with~\eqref{eq:barany_etal_formula} (using the same formulae for the Gamma function as above) and completes its verification.

Another two special cases in which the expression in Theorem~\ref{theo:spherical_polytope_f_vector} can be considerably simplified are given in the following
\begin{proposition}\label{prop:f_vect_of_d+2_points}
For all $d\in\N$ and $k\in\{0,\ldots,d-1\}$, we have
\begin{align*}
&\binom{d+2}{k+1} - \E f_k (C_{d+2}\cap \bS^{d}_+) = \pi^{k-d-1} \frac{d+2}{(k+1)!} \cdot \frac{\sqrt \pi\, \Gamma\left(\frac {d+2}2\right)}{\Gamma\left(\frac{d+3}{2}\right)} \cdot (d+1)^2 A[d,k-1],
\\
&\binom{d+3}{k+1} - \E f_k (C_{d+3}\cap \bS^{d}_+) = \pi^{k-d-1} \frac{d+3}{(k+1)!} \cdot \frac{\sqrt \pi\, \Gamma\left(\frac {d+4}2\right)}{\Gamma\left(\frac{d+3}{2}\right)} \cdot (d+1)^2 A[d,k-1].
\end{align*}
\end{proposition}

Let us also mention that taking $k=1$ in Theorem~\ref{theo:spherical_polytope_f_vector} and observing that $A[d,0]=1$ for all $d\in\N$, we arrive at the following expression for  the number of edges of $C_n\cap \bS^d_+$:
$$
\E f_1(C_n\cap \bS^d_+) = \frac 12 n!\pi^{2-n}\sum_{\substack{s=0,1,\ldots\\ m:=d-2s\geq 2}} \frac{m-1}{(m-2)! (n-m)!} \int_0^\pi (\sin x)^{m-1} x^{n-m} \dd x.
$$

Let us finally observe that using the definition of $B\{n,k\}$ it is not difficult\footnote{Make the change of variables $x= \pi(1 - y/n)$ in the integral $\int_0^\pi (\sin x)^{k-1} x^{n-k} \dd x$.} to show that for every fixed $k\in\N$, we have
$$
\lim_{n\to\infty} \frac{B\{n,k\}}{\pi^n/n!} = 1.
$$
Inserting this into the formula from Theorem~\ref{theo:spherical_polytope_f_vector}, using the recursive property of $A[n,k]$ (see~\eqref{eq:rel_A_n_k_intro}) and evaluating the telescope sum, one can easily re-derive Theorem~\ref{theo:cor}.

\subsection{The expected solid angle}
Next we shall compute the expected spherical volume of $C_n\cap \bS^d_+$, or, which is the same up to a constant factor, the expected solid angle of the cone $C_n$. Let $\alpha(C_n)$ be the angle of the cone $C_n$ (normalized such that the full-space solid angle is $1$). Also, let $\sigma_{d}$ denote the $d$-dimensional surface measure on the sphere $\bS^d$. Recall that $\omega_{d+1} := \sigma_{d}(\bS^d)= 2\pi^{(d+1)/2}/\Gamma(\frac {d+1}2)$.

\begin{theorem}\label{theo:expected_angle}
For all $d\in\N$ and $n\geq d+1$ we have
$$
\E \alpha(C_n) = \frac{ \E \sigma_{d} (C_n\cap \bS_+^d)}{\omega_{d+1}}
=
\frac{n!}{2\pi^{n}} \sum_{\substack{m\in \{d+2,\ldots,n+1\} \\ m \equiv d \Mod{2}}} B\{n+1, m\} (m-1)^2  A[m-2,-1].
$$
\end{theorem}
The asymptotic rate of convergence of $\E \alpha(C_n)$ to $1/2$ (the solid angle of the half-space), as $n\to\infty$ (while $d\in\N$ stays constant) was determined in~\cite[Theorem~7.1]{barany_etal}, where it was shown that
$$
\frac 12 \omega_{d+1}  -  \E \sigma_{d} (C_n\cap \bS_+^d) = C_*(d)n^{-1} + O(n^{-2})
\quad
\text{ and }
\quad
\lim_{n\to\infty} \E f_0(C_n\cap \bS^d_+) = \frac{2C_*(d)}{\omega_{d+1}}
$$
for a certain constant $C_*(d)$ expressed in~\cite[Equation~(22)]{barany_etal} as a multiple integral.  It is not clear how to evaluate this integral.  Comparing the second formula with Theorem~\ref{theo:cor} (where we take $k=0$), we conclude that
$$
C_*(d) = \pi \omega_{d+1} A[d,1]/2.
$$

The proof of Theorem~\ref{theo:expected_angle} is based on an Efron-type identity (see~\eqref{eq:efron}, below) linking the expected angle of $C_n$ to $\E f_0(C_{n+1}\cap \bS^d_+)$. More generally,  Theorem~2.7 of~\cite{convex_hull_sphere} expresses the so-called expected Grassmann angles of the cones $C_n$, $n\in\N$, through their $f$-vectors. Combining this result with Theorem~\ref{theo:spherical_polytope_f_vector}, it is possible to obtain explicit expressions for the expected Grassmann angles of $C_n$. Moreover, Theorem~2.8 of~\cite{convex_hull_sphere} gives asymptotic expressions for the expected Grassmann angles, expected conic intrinsic volumes and expected conic mean projection volumes of the random cone $C_n$, as $n\to\infty$, in terms of certain constants $B_{k,d}$. By combining our Theorem~\ref{theo:cor} with Theorem~2.4 of~\cite{convex_hull_sphere}, we obtain the formula
$$
B_{k,d} = \frac {k!}2 \lim_{n\to\infty} \E f_{k-1} (C_n\cap \bS^d_+) =  \frac{\pi^{k}}{2} A[d,k],
$$
for all $k\in \{1,\ldots,d\}$. This turns all results of~\cite{convex_hull_sphere} that involve the constants $B_{k,d}$ into explicit formulae. We refrain from restating them here.

\subsection{Sylvester problem on the half-sphere}\label{subsec:sylvester}
The classical Sylvester four point problem asks for the probability that  four random points chosen uniformly and independently from some  convex plane region have a convex hull which is a triangle. In the case when the region is a disk (or, more generally, any ellipse), the answer is $35/(12 \pi^2)$. A $d$-dimensional version of this problem was solved by Kingman~\cite{kingman} who computed explicitly the probability that the convex hull of $d+2$ points chosen independently and uniformly from the $d$-dimensional ball is a simplex with $d+1$ vertices. Let us study a similar problem on the half-sphere. Let $U_1,\ldots, U_{d+2}$ be random points sampled uniformly and independently on the upper half-sphere $\bS_+^d$. We ask for the probability $P(d)$ that the spherical convex hull of these points, namely $C_{d+2}\cap \bS^{d}_+$, is a spherical simplex with $d+1$ vertices.

\begin{theorem}\label{theo:sylvester}
For all $d\in\N$ we have
$$
P(d) = \pi^{-(d+1)}\, (d+2) \frac{\sqrt \pi\, \Gamma\left(\frac {d+2}2\right)}{\Gamma\left(\frac{d+3}{2}\right)} \cdot (d+1)^2 A[d,-1].
$$
\end{theorem}
\begin{proof}
Since the number of vertices of the spherical polytope $C_{d+2}\cap \bS^{d}_+$ is either $d+1$ (with probability $P(d)$) or $d+2$ (with probability $1-P(d)$), we have
$$
\E f_0(C_{d+2}\cap \bS^{d}_+) = (d+1)P(d) + (d+2)(1-P(d)) = (d+2) - P(d).
$$
On the other hand, the first formula of Proposition~\ref{prop:f_vect_of_d+2_points} with $k=0$  yields
$$
(d+2) - \E f_0 (C_{d+2}\cap \bS^{d}_+) = \pi^{-(d+1)}(d+2) \frac{\sqrt \pi\, \Gamma\left(\frac {d+2}2\right)}{\Gamma\left(\frac{d+3}{2}\right)} \cdot (d+1)^2 A[d,-1].
$$
Resolving this w.r.t.\ $P(d)$ we arrive at the required formula.
\end{proof}

The first few values of $P(d)$ are given in Table~\ref{tab:sylvester}.
For example, for four points on the two-dimensional half-sphere $\bS^2_+$, the probability that the convex hull is a spherical triangle is $P(2) = \frac{24}{\pi^2}-2 \approx 0.4317$.

\section{Proof of Theorem~\ref{theo:main}}
\subsection{Introduction}
The main difficulty is to prove Theorem~\ref{theo:main} in the case when the codimension $d-\ell$ is even. For future reference, we restate this special case of Theorem~\ref{theo:main} as follows.
\begin{theorem}\label{theo:main_even}
Let $\mathcal Z_d$ be the $d$-dimensional Poisson zero polytope with $d\in\N$. Then, for all $\ell\in \{0,\ldots,d\}$ such that $d-\ell$ is \textbf{even} we have
\begin{equation}
\E f_\ell(\mathcal Z_d) = \frac{\pi^{d-\ell}}{(d-\ell)!} A[d,d-l],
\end{equation}
where $A[d,d-\ell] = [x^{d-\ell}]  Q_d(x)$ and $Q_d(x)$ is the polynomial defined in~\eqref{eq:def_Q_n}.
\end{theorem}
We shall prove Theorem~\ref{theo:main_even} in Sections~\ref{subsec:beta'_poly}--\ref{subsec:comb_identity}. In Section~\ref{subsec:odd_codim} we shall deduce the case of the odd codimension from the Dehn-Sommerville relations.

\subsection{Beta' polytopes}\label{subsec:beta'_poly}
Our proof of Theorem~\ref{theo:main_even} strongly relies on the results of the paper~\cite{beta_polytopes} whose notation we follow.
A random point $X$ in $\R^d$ is said to have the \textit{beta distribution} with parameter $\beta>-1$ if its density is given by
\begin{equation}\label{eq:def_f_beta}
f_{d,\beta}(x)=c_{d,\beta} \left( 1-\left\| x \right\|^2 \right)^\beta\ind_{\{\|x\| <  1\}},\qquad x\in\R^d,\qquad
c_{d,\beta}= \frac{ \Gamma\left( \frac{d}{2} + \b + 1 \right) }{ \pi^{ \frac{d}{2} } \Gamma\left( \beta+1 \right) }.
\end{equation}
Similarly, $X$ has \textit{beta' distribution} with parameter $\beta>d/2$ if its density has the form
\begin{equation}\label{eq:def_f_beta_prime}
\tilde{f}_{d,\beta}(x)=\tilde{c}_{d,\beta} \left( 1+\left\| x \right\|^2 \right)^{-\beta},\qquad
x\in\R^d,\qquad
\tilde{c}_{d,\beta}= \frac{ \Gamma\left( \beta \right) }{\pi^{ \frac{d}{2} } \G\left( \beta - \frac{d}{2} \right) }.
\end{equation}
In order to conform with the notation of~\cite{beta_polytopes}, we usually supply objects and quantities related to the beta' case with the tilde, even though we shall deal almost exclusively with the beta' case here.

Let $\tilde X_1,\tilde X_2,\ldots$ be i.i.d.\ random points in $\R^d$ with density $\tilde f_{d,\beta}$. The convex hull of $\tilde X_1,\ldots,\tilde X_n$ is called the \textit{beta' polytope} and denoted by $\tilde P_{n,d}^{\beta}:=[\tilde X_1,\ldots,\tilde X_n]$. These random polytopes were introduced in the works of Miles~\cite{miles} and Ruben and Miles~\cite{ruben_miles},  and further studied in~\cite{beta_polytopes_temesvari,beta_simplices,bonnet_etal,convex_hull_sphere,beta_polytopes}. In~\cite{beta_polytopes}, expected values of various functionals of these polytopes (including the $f$-vector as well as the internal and external angles) were expressed through quantities of two sorts. The quantities of the first sort, denoted by $\tilde I_{n,k}(\alpha)$, are given by the explicit formula
\begin{equation}\label{eq:I_definition_prime}
\tilde I_{n,k}(\alpha)
=\int_{-\pi/2}^{+\pi/2} \tilde c_{1, \frac{\alpha k + 1}{2}} (\cos x)^{\alpha k -1} \left(\int_{-\pi/2}^x \tilde c_{1,\frac{\alpha+1}{2}} (\cos y)^{\alpha-1} \dint y\right)^{n-k} \dint x,
\end{equation}
see~\cite[Remark 1.17]{beta_polytopes},
and are closely related to the external angles of the beta' polytopes; see~\cite[Theorem~1.16]{beta_polytopes}. The quantities of the second sort, denoted by $\tilde J_{n,k}(\beta)$, are defined as follows. Let $\tilde Z_1,\ldots,\tilde Z_n$ be $n$ independent random points in $\R^{n-1}$ distributed according to the density $\tilde f_{n-1,\beta}$, where $\beta>\frac{n-1}{2}$. Then, $\tilde J_{n,k}(\beta)$ is the expected internal angle of the simplex $[\tilde Z_1,\ldots, \tilde Z_n]$ at its face $[\tilde Z_1,\ldots,\tilde Z_{k}]$, for $k\in \{1,\ldots, n\}$. By definition, $\tilde J_{n,n}(\beta) = 1$.

For the purposes of the present paper, it will be more convenient to work with the quantities
$$
\tilde \II_{n,k}(\beta) := \binom nk \tilde I_{n,k}(\beta)
\quad \text{ and }\quad
\tilde \JJ_{n,k}(\beta) := \binom nk \tilde J_{n,k}(\beta),
\qquad k\in \{1,\ldots,n\}.
$$
Note that, by definition,  $\tilde \JJ_{n,k}(\beta)$ is the expected sum of internal angles at all $k$-vertex faces $[\tilde Z_{i_1},\ldots, \tilde Z_{i_k}]$ of the beta' simplex $[\tilde Z_1,\ldots,\tilde Z_n]\subset \R^{n-1}$ with $n$ vertices. On the other hand, $\tilde \II_{n,k}(2\beta-n+1)$ is the expected sum of external angles at all $k$-vertex faces of the same random simplex; see~\cite[Theorem~1.16]{beta_polytopes}.

\subsection{Expected internal angle sums}
It follows from~\eqref{eq:E_f_ell_Z_d_Poisson_Process} that to prove Theorem~\ref{theo:main_even} (which is trivial for $\ell=d$) it suffices to show that for all $d\in\N$ and all \textbf{odd} $k = d-\ell-1 \in \{1,\ldots,d-1\}$, we have
\begin{equation}\label{eq:E_f_k_Pi_d_1}
\E f_{k}(\conv \Pi_{d,1})
=
\frac{\pi^{k+1}}{(k+1)!} A[d,k+1],
\end{equation}
where $\conv \Pi_{d,1}$ is the convex hull of the Poisson process $\Pi_{d,1}$ defined in Section~\ref{subsec:convex_half_sphere}.
The starting point of our proof is the following explicit formula derived in~\cite[Theorem 1.21]{beta_polytopes} (where we have to put $\alpha=1$, see Section~\cite[Section~1.5]{beta_polytopes} for details):
\begin{equation}\label{eq:lim_E_f_k_C_n}
\E f_{k}(\conv \Pi_{d,1})  = \sum_{\substack{s=0,1,\ldots\\n:=d-2s\geq k+1}} \frac {2}{n} \pi^n \tilde c_{1, \frac{n+1}{2}} \tilde \JJ_{n,k+1}\left(\frac n2\right),
\end{equation}
for all $k\in \{0,\ldots,d-1\}$. The main contribution of the present paper is the evaluation of the expected internal angle sums $\tilde \JJ_{n,k+1}(n/2)$.

\begin{proposition}\label{prop:internal_explicit}
For all $n\in\N$ and \textbf{even} $k\in \{1,\ldots, n\}$, the expected sum of internal angles at faces with $k$ vertices of the beta' simplex $\tilde P_{n,n-1}^{n/2}\subset \R^{n-1}$ with $n$ vertices and $\beta=n/2$ is given by
\begin{equation}\label{eq:internal_explicit}
\tilde \JJ_{n,k}\left(\frac n2\right)
=
\frac{\pi^{k-n}}{k!}\cdot \frac{n}{2 \tilde c_{1, \frac{n+1}{2}}} \cdot (n-1)^2 A[n-2,k-2].
\end{equation}
\end{proposition}
\begin{remark}
Proposition~\ref{prop:internal_explicit} is true for odd $k$, too, but we shall be able to prove this only in Section~\ref{subsec:removing_parity},  after Theorem~\ref{theo:main} has been established without parity restrictions.
\end{remark}

\begin{remark}
In the subsequent publications~\cite{kabluchko_angles,kabluchko_algorithm,kabluchko_formula} we shall develop the methods of the present paper to compute $\tilde \JJ_{n,k}(\beta)$ and other related quantities for all $\beta>\frac{n-1}{2}$. The special case $\beta=\frac n2$ studied here is distinguished by especially nice combinatorial properties.
\end{remark}

\begin{proof}[Proof of Theorem~\ref{theo:main_even} given Proposition~\ref{prop:internal_explicit}]
Using~\eqref{eq:def_A_n_k}, it is easy to check that $A[n,k+1]-A[n-2,k+1] = (n-1)^2A[n-2,k-1]$ for all $n\in \{2,3,\ldots\}$ and $k\in\Z$; see, e.g.,\  Lemma~\ref{lem:rec_first_kind}.
Replacing $k$ by $k+1$ and using this relation, we can write~\eqref{eq:internal_explicit} as
$$
\tilde \JJ_{n,k+1}\left(\frac n2\right) = \frac{\pi^{k+1-n}}{(k+1)!}\cdot \frac{n}{2 \tilde c_{1, \frac{n+1}{2}}} \cdot (A[n,k+1] - A[n-2,k+1]),
$$
for all odd $k\in \{1,\ldots,n-1\}$.

We need to prove~\eqref{eq:E_f_k_Pi_d_1} for odd $k \in \{1,\ldots,d-1\}$. Plugging the above formula for $\tilde \JJ_{n,k+1}(n/2)$ into~\eqref{eq:lim_E_f_k_C_n}, we obtain that for all such $k$,
$$
\E f_{k}(\conv \Pi_{d,1})
=
\sum_{\substack{s=0,1,\ldots\\n:=d-2s\geq k+1}} \frac{\pi^{k+1}}{(k+1)!}\cdot  (A[n,k+1] - A[n-2,k+1])
=
\frac{\pi^{k+1}}{(k+1)!} A[d,k+1]
$$
because the last term in the telescope sum, which is either $-A[k-1,k+1]$ or $-A[k,k+1]$, vanishes. This establishes~\eqref{eq:E_f_k_Pi_d_1}.
\end{proof}

\subsection{System of equations for expected internal angles}
Our main task is thus to prove Proposition~\ref{prop:internal_explicit}. As a first step, we shall provide a system of relations between the quantities  $\tilde \II_{n,k}(\beta)$ and $\tilde \JJ_{n,k}(\beta)$ which leads to a recursive algorithm for computing $\tilde \JJ_{n,k}(\beta)$.

\begin{proposition}\label{prop:relations}
For every $n\in \{2,3,\ldots\}$, $k\in \{1,\ldots,n-1\}$ and for every $\beta>(n-1)/2$ we have
\begin{align}
&\sum_{\substack{s=0,1,\ldots\\ n-2s\geq k}} \tilde \II_{n,n-2s}(2\beta-n+1) \tilde \JJ_{n-2s,k} (\beta-s) = \frac 12 \binom nk,\label{eq:relation_U_V}\\
&\sum_{\substack{s=0,1,\ldots\\ n-2s-1\geq k}} \tilde \II_{n,n-2s-1}(2\beta-n+1) \tilde \JJ_{n-2s-1,k} \left(\beta - s - \frac 12\right) = \frac 12 \binom nk.\label{eq:relation_U_V_one_more}
\end{align}
\end{proposition}
\begin{proof}
Given a  $d$-dimensional polyhedral cone $C$, we denote by  $\upsilon_{0}(C),\ldots,\upsilon_d(C)$  its conic intrinsic volumes. For their definition and a review of their properties we refer to~\cite{ALMT14,AmelunxenLotzDCG17} (whose notation we follow) and to~\cite[Section~6.5]{SW08} (where slightly different notation is used).  Here we shall need only the
\textit{Gauss-Bonnet relation}~\cite[Equation~(5.3)]{ALMT14} which states that
$$
\sum_{\substack{s\geq 0\\ j := d-2s\geq 0}} \upsilon_j(C) = \sum_{\substack{s\geq 0\\ j := d-2s-1\geq 0}} \upsilon_j(C) = \frac 12
$$
for every $d$-dimensional polyhedral cone $C$ that is not a linear subspace.

Consider the $(n-1)$-dimensional beta' simplex $\tilde P_{n,n-1}^\beta$ defined as the convex hull $[\tilde X_1,\ldots,\tilde X_{n}]$ of $n$ independent random points $\tilde X_1,\ldots,\tilde X_{n}$ having the probability density $\tilde f_{n-1,\beta}$ on $\R^{n-1}$. The tangent cone at its $k$-vertex face $G=[\tilde X_1,\ldots,\tilde X_k]$ is defined as
$$
\tilde T_{n,k}^\beta:= \{v\in\R^{n-1}: \text{ there exists } \eps>0 \text{ such that } g_0 + \eps v\in \tilde P_{n,n-1}^\beta\},
$$
where $g_0$ is any point in the relative interior of $G$, for example $g_0= (\tilde X_1+\ldots+\tilde X_k)/k$.

The expected conic intrinsic volumes of the tangent cone were computed in~\cite[Theorem~1.18]{beta_polytopes}: For all $k\in \{1,\ldots,n-1\}$ and $j\in \{k-1,\ldots,n-1\}$ we have
\begin{equation}\label{eq:E_intrinsic_tangent_cone}
\E \upsilon_j(\tilde T_{n,k}^\beta) =\frac 1 {\binom {n}{k}}\tilde \II_{n,j+1}(2\beta-n+1) \tilde \JJ_{j+1,k}\left(\beta - \frac{n-1-j}{2}\right).
\end{equation}
For $j\notin \{k-1,\ldots,n-1\}$ we have $\E \upsilon_j(\tilde T_{n,k}^\beta)=0$. In particular, all intrinsic volumes with $j<k-1$ vanish, which is due to the fact that the lineality space of the tangent cone, defined as the intersection of $\tilde T_{n,k}^\beta$ with $-\tilde T_{n,k}^\beta$, coincides with the affine hull of $G$ shifted to the origin and has dimension $k-1$.

Applying the Gauss-Bonnet relation to the tangent cone $\tilde T_{n,k}^\beta$ and taking the expectation, we arrive at the required relation~\eqref{eq:relation_U_V}.
\end{proof}

In view of the interpretation of $\tilde \II_{n,k}(\beta)$ and $\tilde \JJ_{n,k}(\beta)$ as expected sums of internal/external angles, relations~\eqref{eq:relation_U_V} and~\eqref{eq:relation_U_V_one_more} can be seen as a stochastic version of McMullen's non-linear angle-sum relations~\cite{mcmullen} in the setting of beta' polytopes.

The above proposition leads to a recursive algorithm which can be used to compute the quantities $\tilde \JJ_{n,k}(\beta)$, both numerically and exactly. First, recall  that $\tilde \JJ_{n,n}(\beta)=1$ by definition. Separating in~\eqref{eq:relation_U_V} the term with $s=0$ and noting that $\tilde \II_{n,n}(2\beta-n+1)=1$ by~\eqref{eq:I_definition_prime}, we can write
$$
\tilde \JJ_{n,k} (\beta) =  \frac 12 \binom nk  - \sum_{\substack{s=1,2,\ldots\\ n-2s\geq k}} \tilde \II_{n,n-2s}(2\beta-n+1) \tilde \JJ_{n-2s,k} (\beta-s)
$$
for $k\in \{1,\ldots,n-1\}$.  This gives an expression for $\tilde \JJ_{n,k} (\beta)$ in terms of the quantities of the form $\tilde \JJ_{\ell,k} (\alpha)$ with $\ell<n$ and the quantities of the form $\tilde \II_{n,k}(\alpha)$ for which we have explicit expression~\eqref{eq:I_definition_prime}. Proceeding recursively, we can compute $\tilde \JJ_{n,k}(\beta)$.

Using this algorithm together with~\eqref{eq:lim_E_f_k_C_n} we computed exactly the expected $f$-vectors of the Poisson zero polytopes in low dimensions.   Then, we guessed the formula stated in Theorem~\ref{theo:main_even} by the method of trials and errors.

\subsection{Uniqueness of the solution and reduction to a combinatorial identity}
To prove Proposition~\ref{prop:internal_explicit}, we shall proceed as follows. First, we shall observe that the system of linear equations~\eqref{eq:relation_U_V}, which is triangular with $1$'s on the diagonal,  determines the unknown quantities $\tilde \JJ_{n,k} (\beta)$ uniquely. This will be stated more precisely in the next proposition. Thus, in order to prove Proposition~\ref{prop:internal_explicit}, it suffices to check that equation~\eqref{eq:relation_U_V} continues to hold if we replace $\tilde \JJ_{n,k} (n/2)$ by their conjectured values given by the right-hand side of~\eqref{eq:internal_explicit}. This reduces Proposition~\ref{prop:internal_explicit} to certain combinatorial identity which will be verified in Section~\ref{subsec:comb_identity}.

First of all, we are interested in the particular case $\beta= n/2$ of the above setting, in which case equation~\eqref{eq:relation_U_V} simplifies to
\begin{equation}\label{eq:relation_U_V_simple}
\sum_{\substack{s=0,1,\ldots\\ n-2s\geq k}} \tilde \II_{n,n-2s}(1) \tilde \JJ_{n-2s,k} \left(\frac n2-s\right) = \frac 12 \binom nk,
\end{equation}
for all $n\in \{2,3,\ldots\}$, $k\in \{1,\ldots,n-1\}$.
\begin{proposition}\label{prop:unique1}
Consider the following system of linear equations for the unknown quantities $v_{n,k}$, where $n\in \{2,3,\ldots\}$ and $k$ is an even number in $\{1,\ldots,n\}$:
\begin{align}
&\sum_{\substack{s=0,1,\ldots\\ n-2s\geq k}} \tilde \II_{n,n-2s}(1) v_{n-2s,k} = \frac 12 \binom nk,
\text{ for all } n\in \{2,3,\ldots\}, \; k\in 2\N,\; k<n, \label{eq:systA1}\\
&v_{n,n}=1, \text{ for all even } n\in \{2,4,6,\ldots\}.\label{eq:systA2}
\end{align}
Then, the unique solution of this system is given by $v_{n,k}=\tilde \JJ_{n,k} (n/2)$, for all $n\in \{2,3,\ldots\}$ and all even $k\in\{1,\ldots,n\}$.
\end{proposition}
\begin{proof}
Since $v_{n,k}=\tilde \JJ_{n,k} (n/2)$ is indeed a solution according to~\eqref{eq:relation_U_V_simple}, it remains to show that the solution is unique. This will be done by induction. For the base case $n=2$, note that the quantity $v_{2,2}=1$
is determined uniquely by~\eqref{eq:systA2}.
Assume now that $n\in \{3,4,\ldots\}$ and that we have shown that the quantities $v_{m,\ell}$ are uniquely determined by~\eqref{eq:systA1}, \eqref{eq:systA2} for all $m\in \{2,3,\ldots,n-1\}$ and all even $\ell\in \{1,\ldots,m\}$. We are going to show that $v_{n,k}$ are determined uniquely for all  even $k\in \{1,\ldots,n\}$. If $n$ is even and $n=k$, then $v_{n,n}=1$ by~\eqref{eq:systA2}. So, let $k\in 2\N$ be even with $k<n$. Then, separating the term with $s=0$ in~\eqref{eq:systA1} and observing that $\tilde \II_{n,n}(1)=1$, we can express $v_{n,k}$ through $v_{m,\ell}$'s
with $m=n-2s$ strictly smaller than $n$, thus completing the induction.
\end{proof}


In view of the above, in order to prove Proposition~\ref{prop:internal_explicit}, it suffices to check that for all $n\in \{2,3,\ldots\}$, and all even $k\in \{1,\ldots,n-1\}$
\begin{equation}\label{eq:need_to_prove1}
\sum_{\substack{s=0,1,\ldots\\ n-2s\geq k}} \left(\tilde \II_{n,n-2s}(1)\cdot
\frac{\pi^{k-n+2s}}{k!}\cdot \frac{n-2s}{\tilde c_{1, \frac{n-2s+1}{2}}} \cdot (n-2s-1)^2 \cdot A[n-2s-2,k-2]\right) =  \binom nk.
\end{equation}
Sections~\ref{subsec:B_n_k_def}--\ref{subsec:comb_identity} are devoted to the proof of~\eqref{eq:need_to_prove1}.

\subsection{Definition of \texorpdfstring{$B\{n,k\}$}{B\{n,k\}}}\label{subsec:B_n_k_def}
For $n\in \N$ and $k\in\{1,\ldots,n\}$  consider the numbers
\begin{equation}\label{eq:def_B_n_k}
B\{n,k\} := \frac {1}{(k-1)!(n-k)!} \int_0^\pi (\sin x)^{k-1} x^{n-k} \dint x.
\end{equation}
Note that $B\{n,1\} = \pi^n/n!$.
The values of $B\{n,k\}$ for small $n$ and $k$ are given in Table~\ref{tab:B_n_k}.
Let us extend this definition by putting
\begin{equation}\label{eq:def_B_n_k_complement}
B\{n,k\}
:=
\begin{cases}
\pi^n/n!, &\text{ for all } n\in\N, \; k=0,\\
0, & \text{ for all } n\in\N, \; k\in \{n+1,n+2,\ldots\}.
\end{cases}
\end{equation}
The next lemma expresses $\tilde \II_{n,k}(1)$ through $B\{n,k\}$.
\begin{lemma}\label{lem:I_n_k_B_n_k}
For all $n\in\N$ and $k\in \{1,\ldots, n\}$ we have
$$
\tilde \II_{n,k}(1) = \binom{n}{k}\tilde I_{n,k}(1) = \frac {n!}k  \tilde c_{1, \frac{k+1}{2}} \pi^{k-n} B\{n, k\}.
$$
\end{lemma}
\begin{proof}
Using~\eqref{eq:I_definition_prime} with $\alpha=1$, the fact that $\tilde c_{1,1} = 1/\pi$ (see~\eqref{eq:def_f_beta_prime}) and finally the variable change $x= \varphi-\frac \pi2$, we obtain
\begin{align*}
\tilde I_{n,k}(1)
=
\int_{-\pi/2}^{+\pi/2} \tilde c_{1,\frac{k+1}{2}} (\cos x)^{k-1} \left(\frac{x}{\pi} + \frac {1}{2}\right)^{n-k} \dint x
=
\tilde c_{1,\frac{k+1}{2}} \pi^{k-n} \int_{0}^{\pi}  (\sin \varphi)^{k-1} \varphi^{n-k} \dint \varphi,
\end{align*}
which proves the claim after recalling that $\tilde \II_{n,k}(1)= \binom nk \tilde I_{n,k}(1)$.
\end{proof}

In view of Lemma~\ref{lem:I_n_k_B_n_k}, we can rewrite~\eqref{eq:need_to_prove1} in the following form:
\begin{equation}\label{eq:need_to_prove2}
\sum_{\substack{s=0,1,\ldots \\ n-2s\geq k}} B\{n,n-2s\}  (n-2s-1)^2 A[n-2s-2,k-2] = \frac{\pi^{n-k}}{(n-k)!},
\end{equation}
for all $n\in \{2,3,\ldots\}$ and all even $k\in \{1,\ldots,n-1\}$. Our task is to prove~\eqref{eq:need_to_prove2}.

\subsection{Recurrence relations for \texorpdfstring{$A[n,k]$}{A[n,k]} and \texorpdfstring{$B\{n,k\}$}{B\{n,k\}}}
First we establish recurrence relations for $A[n,k]$ and $B\{n,k\}$. These are similar to the relations satisfied by the Stirling numbers $\stirling{n}{k}$ and $\stirlingb{n}{k}$; see~\eqref{eq:stirling_relations}.

\begin{lemma}\label{lem:rec_first_kind}
For all $n\in\N_0$ and all $k\in\Z$ we have
\begin{equation}\label{eq:A_rec}
A[n+2,k] - A[n,k] = (n+1)^2 A[n,k-2].
\end{equation}
\end{lemma}
\begin{proof}
Let $k\in\Z$ be even. By definition of $A[n,k]$, see~\eqref{eq:def_A_n_k} and~\eqref{eq:def_Q_n}, we have
\begin{align*}
A[n+2,k] &=  [x^k] \Big((1+ (n+1)^2x^2) (1+(n-1)^2 x^2) (1+(n-3)^2 x^2) \ldots\Big),\\
A[n,k] &=  [x^k] \Big((1+ (n-1)^2x^2) (1+(n-3)^2 x^2) (1+(n-5)^2 x^2) \ldots\Big).
\end{align*}
Subtracting these identities, we obtain
\begin{align*}
A[n+2,k] - A[n,k]
&=
[x^k] \Big((n+1)^2x^2(1+(n-1)^2 x^2) (1+(n-3)^2 x^2) \ldots\Big)\\
&=
(n+1)^2 A[n,k-2],
\end{align*}
thus proving the claim. The proof in the case when $k$ is odd is similar and does not use any special properties of the functions $\tanh$ and $\cotanh$.
\end{proof}

Recall that the numbers $B\{n,k\}$ were defined in~\eqref{eq:def_B_n_k} and~\eqref{eq:def_B_n_k_complement}.
\begin{lemma}\label{lem:rec_second_kind}
For all $n\in\N$ and $k\in \{2,3,\ldots\}$, we have
\begin{equation}\label{eq:B_rec}
B\{n,k-2\} - B\{n,k\} = (k-1)^2 B\{n+2,k\}.
\end{equation}
\end{lemma}
\begin{proof}
\textit{Case 1.}
To begin with, let us assume that $k\in \{3,\ldots,n\}$. Integrating by parts, we obtain
\begin{align*}
B\{n,k\}
&=
\frac {1}{(k-1)!(n-k)!} \int_0^\pi (\sin x)^{k-1} \dint \left(\frac{x^{n-k+1}}{n-k+1}\right)\\
&=
-\frac {1}{(k-1)!(n-k+1)!} \int_0^\pi x^{n-k+1} (k-1) (\sin x)^{k-2} (\cos x) \dint x.
\end{align*}
Applying partial integration for the second time, we arrive at
\begin{align*}
B\{n,k\}
&=
-\frac {1}{(k-2)!(n-k+1)!} \int_0^\pi (\sin x)^{k-2} (\cos x) \dint \left(\frac{x^{n-k+2}}{n-k+2}\right)\\
&=
\frac {1}{(k-2)!(n-k+2)!} \int_0^\pi x^{n-k+2} \frac {\dint }{\dint x}\Big((\sin x)^{k-2} (\cos x)\Big) \dint x\\
&=
\frac {1}{(k-2)!(n-k+2)!} \int_0^\pi x^{n-k+2} \Big((k-2)(\sin x)^{k-3}(\cos x)^2 - (\sin x)^{k-1}\Big) \dint x.
\end{align*}
Observe that we required $k\geq 3$ because, for $k=2$, the term $x^{n-k+2}(\sin x)^{k-2}(\cos x)$ appearing in the partial integration formula does not vanish at $x=\pi$.

Using the identity $\cos^2 x= 1-\sin^2 x$ and simplifying, we arrive at
\begin{multline*}
B\{n,k\}
=
\frac {(k-2)}{(k-2)!(n-k+2)!} \int_0^\pi x^{n-k+2} (\sin x)^{k-3} \dint x
\\-
\frac {(k-1)}{(k-2)!(n-k+2)!} \int_0^\pi x^{n-k+2} (\sin x)^{k-1} \dint x.
\end{multline*}
We now easily recognize that the first term on the right-hand side is $B\{n,k-2\}$, whereas the second term is $(k-1)^2 B\{n+2,k\}$. Thus, we proved that $B\{n,k\} = B\{n,k-2\} - (k-1)^2 B\{n+2,k\}$ for $k\in \{3,\ldots,n\}$. 

\vspace*{2mm}
\noindent
\textit{Case 2.}
If $k\geq  n+3$, then all terms in~\eqref{eq:B_rec} vanish by definition.

\vspace*{2mm}
\noindent
\textit{Case 3.} If $k=n+2$ or $k=n+1$, then $B\{n,k\}=0$ by definition and we need to verify that
$$
B\{n,n\} = (n+1)^2B\{n+2,n+2\}
\qquad \text{ and } \qquad
B\{n,n-1\} = n^2 B\{n+2,n+1\}
$$
for all $n\in\N$.
The second identity holds for $n=1$ since $B\{1,0\}=\pi = B\{3,2\}$, and we ignore this case in the following.  Recalling the definition of $B\{n,k\}$ stated in~\eqref{eq:def_B_n_k}, we can write these identities as
\begin{align}
&\int_0^\pi (\sin x)^{n+1} \dd x = \frac {n}{n+1} \int_0^\pi (\sin x)^{n-1} \dd x,
\qquad n\in\N,\label{eq:verify_sin_int_1}\\
\qquad
&\int_0^\pi x (\sin x)^{n} \dd x = \frac {n-1}{n} \int_0^\pi x (\sin x)^{n-2} \dd x,
\qquad n\geq 2.\label{eq:verify_sin_int_2}
\end{align}
To verify~\eqref{eq:verify_sin_int_1}, we  use partial integration as follows:
$$
\int_0^\pi (\cos^2 x) (\sin x)^{n-1} \dd x
=
\frac 1n  \int_0^\pi (\cos x)  \dd (\sin x)^n
=
\frac 1n  \int_0^\pi (\sin x) (\sin x)^{n}  \dd x.
$$
Replacing $\cos^2 x$ by $1-\sin^2 x$ on the left-hand side, we arrive at~\eqref{eq:verify_sin_int_1}. To verify~\eqref{eq:verify_sin_int_2}, write
\begin{multline*}
\int_0^\pi x (\cos^2 x) (\sin x)^{n-2} \dd x
=
\frac 1{n-1}  \int_0^\pi (x\cos x)  \dd (\sin x)^{n-1}\\
=
\frac 1{n-1}  \int_0^\pi (x\sin x-\cos x) (\sin x)^{n-1}  \dd x
=
\frac 1{n-1}  \int_0^\pi x(\sin x)^{n}  \dd x
\end{multline*}
because $\int_0^\pi (\cos x) (\sin x)^{n-1}  \dd x=0$.  Replacing $\cos^2 x$ by $1-\sin^2 x$ on the left-hand side, we arrive at~\eqref{eq:verify_sin_int_2}.

\vspace*{2mm}
\noindent
\textit{Case 4.}
Let $k=2$. If $n=1$, identity~\eqref{eq:B_rec} takes the form $B\{1,0\} = B\{3,2\}$, which is true because both terms are equal to $\pi$. In the case when $n\geq 2$,  we need to verify the identity
$
\pi^n/n! = B\{n,2\} + B\{n+2,2\}
$,
or, after recalling~\eqref{eq:def_B_n_k} and multiplying by $n!$,
$$
\pi^n = \int_0^\pi (\sin x) (n(n-1)x^{n-2} + x^n) \dd x,\qquad n\geq 2.
$$
This is an easy exercise in partial integration:
\begin{multline*}
\int_0^\pi (\sin x) n(n-1)x^{n-2} \dd x
=
\int_0^\pi n (\sin x)  \dd x^{n-1}
=
- \int_0^\pi n (\cos x) x^{n-1} \dd x
=
- \int_0^\pi (\cos x)  \dd x^{n}\\
=
- (\cos x) x^n \Big |_0^\pi -  \int_0^\pi (\sin x) x^{n} \dd x
=
\pi^n - \int_0^\pi (\sin x) x^{n} \dd x .
\end{multline*}
\end{proof}

\subsection{The basic combinatorial identity}\label{subsec:comb_identity}
In the next lemma we prove~\eqref{eq:need_to_prove2}, thereby completing the proof of Proposition~\ref{prop:internal_explicit} and Theorem~\ref{theo:main_even}.
\begin{lemma}\label{lemma:identity_A_B_1}
For all $n\in\N$ and all \textbf{even} $k\in \{1,2,\ldots,n-1\}$,
\begin{equation}\label{eq:identity}
\sum_{\substack{s=0,1,\ldots \\ n-2s\geq k}} B\{n,n-2s\}  (n-2s-1)^2 A[n-2s-2,k-2] = \frac{\pi^{n-k}}{(n-k)!}.
\end{equation}
\end{lemma}
\begin{remark}
Another formula of the same type will be established in Lemma~\ref{lemma:identity_A_B_2}, where we also remove the parity restriction.
\end{remark}
\begin{proof}[Proof of Lemma~\ref{lemma:identity_A_B_1}]
We argue by induction, assuming the identity for some $n$ and proving it for $n+2$. 

\vspace*{2mm}
\noindent
\textit{Base cases.} We start by verifying the cases $n=3$ and $n=4$ (because for $n=1$ and $n=2$ the set of admissible $k$'s is empty).


\vspace*{2mm}
\noindent
\textit{Case $n=3$.} Then, $k=2$ and~\eqref{eq:identity} turns into
$
B\{3,3\}  2^2  A[1,0] = \pi,
$
which is true because $A[1,0]=1$ and $B\{3,3\} = \pi/4$.

\vspace*{2mm}
\noindent
\textit{Case $n=4$.} Then, $k=2$ and~\eqref{eq:identity} turns into
$$
B\{4,4\} 3^2 A[2,0] +  B\{4,2\} 1^2 A[0,0] = \frac{\pi^2}{2},
$$
which is true because $A[2,0]=A[0,0]=1$, while $B\{4,4\}=2/9$ and $B\{4,2\} = \pi^2/2-2$.

\vspace*{2mm}
\noindent
\textit{Induction assumption.}
Assume that identity~\eqref{eq:identity} holds for some $n\in\{3,4,\ldots\}$ and all even $k\in \{1,\ldots, n-1\}$.
By Lemma~\ref{lem:rec_first_kind}, we have
$$
(n-2s-1)^2 A[n-2s-2,k-2] = A[n-2s,k] - A[n-2s-2,k],
$$
so that we can write the induction assumption in the form
$$
\sum_{\substack{s=0,1,\ldots \\ n-2s\geq k}} B\{n,n-2s\} \Big( A[n-2s,k] - A[n-2s-2,k] \Big) = \frac{\pi^{n-k}}{(n-k)!}
$$
or, more conveniently,
\begin{equation}\label{eq:ind_ass_convenient}
\sum_{\substack{s=0,1,\ldots \\ n-2s\geq k}} B\{n,n-2s\}  A[n-2s,k] =
\frac{\pi^{n-k}}{(n-k)!} + \sum_{\substack{s=0,1,\ldots \\ n-2s\geq k}} B\{n,n-2s\} A[n-2s-2,k],
\end{equation}
for all even $k\in \{1,\ldots,n-1\}$.

\vspace*{2mm}
\noindent
\textit{Induction step.}
We need to prove that  identity~\eqref{eq:identity} holds with $n$ replaced by $n+2$, that is
\begin{equation}\label{eq:induction_need_to_prove}
S:=\sum_{\substack{s=0,1,\ldots \\ n+2-2s\geq k}} B\{n+2,n+2-2s\}  (n-2s+1)^2 A[n-2s,k-2] = \frac{\pi^{n+2-k}}{(n+2-k)!},
\end{equation}
for all even $k\in \{1,\ldots,n+1\}$.
By Lemma~\ref{lem:rec_second_kind}, for all $s$ such that $n+2-2s\geq k\geq 2$ we have
$$
(n-2s+1)^2 B\{n+2,n+2-2s\} = B\{n,n-2s\} - B\{n,n-2s+2\}.
$$
Inserting this into the above definition of $S$, we obtain
\begin{align*}
S
&=
\sum_{\substack{s=0,1,\ldots \\ n+2-2s\geq k}}  \left(B\{n,n-2s\} - B\{n,n-2s+2\}\right) A[n-2s,k-2]\\
&=
\sum_{\substack{s=0,1,\ldots \\ n-2s\geq k-2}}  B\{n,n-2s\} A[n-2s,k-2]
-
\sum_{\substack{s=0,1,\ldots \\ n+2-2s\geq k}}  B\{n,n-2s+2\} A[n-2s,k-2].
\end{align*}
Let  $k\neq 2$. To the first sum we apply the induction assumption~\eqref{eq:ind_ass_convenient} with $k$ replaced by $k-2$ (which is an even number in the range $\{1,\ldots,n-1\}$):
\begin{align*}
S
=
\frac{\pi^{n-(k-2)}}{(n-(k-2))!}
&
+ \sum_{\substack{s=0,1,\ldots \\ n-2s\geq k-2}} B\{n,n-2s\} A[n-2s-2,k-2]\\
&-
\sum_{\substack{s=0,1,\ldots \\ n-2s\geq k-2}}  B\{n,n-2s+2\} A[n-2s,k-2].
\end{align*}
Introducing the new summation index $s':=s-1$ in the second sum and leaving the first sum unchanged, we obtain
\begin{align*}
S
=
\frac{\pi^{n-(k-2)}}{(n-(k-2))!}
&
+ \sum_{\substack{s=0,1,\ldots \\ n-2s\geq k-2}} B\{n,n-2s\} A[n-2s-2,k-2]\\
&-
\sum_{\substack{s'=-1,0,\ldots \\ n-2s'\geq k}}  B\{n,n-2s'\} A[n-2s'-2,k-2].
\end{align*}
The sums on the right-hand side differ by just two terms corresponding to $s'=-1$ (in the second sum)  and $s$ such that $n-2s\in \{k-1,k-2\}$ (in the first sum). The term with $s'=-1$ is $B\{n,n+2\} A[n,k-2]$, which vanishes by definition.
The term in the first sum for which $n-2s\in \{k-1,k-2\}$ also vanishes because then $n-2s-2\in \{k-3, k-4\}$ and consequently $A[n-2s-2,k-2] = 0$. So, the sums cancel each other and we are left with
$$
S
=
\frac{\pi^{n+2-k}}{(n+2-k)!},
$$
which verifies~\eqref{eq:induction_need_to_prove}.   To complete the induction, it remains to check the case $k=2$. Since $A[n-2s,0]=1$, we have
$$
S=
\sum_{\substack{s=0,1,\ldots \\ n-2s\geq 0}}  B\{n,n-2s\}
-
\sum_{\substack{s=0,1,\ldots \\ n-2s\geq 0}}  B\{n,n-2s+2\}.
$$
Again, the sums differ by just two terms. One of them  is $-B\{n,n+2\}=0$. The other term is $B\{n,0\}=\pi^n/n!$ (if $n$ is even) or $B\{n,1\} = \pi^n/n!$ (if $n$ is odd). In both cases, we have $S= \pi^n/n!$, which completes the induction.
\end{proof}

The proof of Theorem~\ref{theo:main_even} is thus complete.

\begin{remark}
For odd $k$, there is just one place where the above proof of Lemma~\ref{lemma:identity_A_B_1} breaks down. Namely, for odd $k$ we would need to verify the case $k=1$ (instead of $k=2$). The corresponding  expression for $S$ involves the non-trivial and ``ugly'' terms $A[n-2s,-1]$ instead of the trivial values $A[n-2s,0]=1$.
\end{remark}

\subsection{Proof of Theorem~\ref{theo:main} for odd codimension \texorpdfstring{$d-\ell$}{d-l}}\label{subsec:odd_codim}
Let us fix $d\in\N$ and introduce the quantities
$$
g_{\ell} := \frac{(d-\ell)!}{\pi^{d-\ell}} \E f_\ell (\mathcal Z_d),\qquad \ell\in \{0,\ldots,d\}.
$$
We already know from Theorem~\ref{theo:main_even} that
\begin{equation}\label{eq:g_ell_even_codim}
g_\ell = A[d,d-\ell] = [x^{d-\ell}]Q_d(x)
\end{equation}
provided $d-\ell$ is even. Our aim is to identify the numbers $g_\ell$ in the case when $d-\ell$ is odd. More precisely, we claim that
\begin{equation}\label{eq:g_l_need}
g_\ell
=
A[d,d-\ell]
=
\begin{cases}
[x^{d-\ell}] \left(\tanh \left(\frac{\pi}{2x}\right)  \cdot  Q_d(x)\right), &\text{ if $\ell$ is odd and $d$ is even},\\
[x^{d-\ell}] \left(\cotanh \left(\frac{\pi}{2x}\right) \cdot Q_d(x)\right), &\text{ if $\ell$ is even and $d$ is odd}.
\end{cases}
\end{equation}
To prove this, we shall use the Dehn-Sommerville relations~\eqref{eq:dehn_sommerville_Ef} which take the form
\begin{equation}\label{eq:dehn_sommer_g_l}
g_\ell = \sum_{i=0}^\ell (-1)^i g_i \frac{\pi^{\ell-i}}{(\ell-i)!},
\end{equation}
for all $l\in \{0,\ldots,d\}$.
Multiplying~\eqref{eq:dehn_sommer_g_l} by $x^\ell$ and taking the sum over $\ell\in \{0,\ldots,d\}$, we obtain the relation
$$
\sum_{\ell = 0}^d g_\ell x^\ell = \sum_{\ell = 0}^d  \sum_{i=0}^\ell \left((-1)^i g_i x^i\right) \left(\frac{(\pi x)^{\ell-i}}{(\ell-i)!}\right)
=
\left(\sum_{i = 0}^d (-1)^i g_i x^i \right)
\eee^{\pi x} + x^{d+1} R(x),
$$
where $R(x) = a_{0} +  a_{1}x + \ldots$ is some power series.
Introducing the generating polynomials
$$
G_{\text{even}}(x) := \sum_{\substack{r=0,2,4,\ldots\\r\leq d}} g_r x^r,
\qquad
G_{\text{odd}}(x) := \sum_{\substack{r=1,3,5,\ldots\\ r\leq d}} g_r x^r,
$$
we can write this as
$$
G_{\text{even}}(x) + G_{\text{odd}}(x) = (G_{\text{even}}(x) - G_{\text{odd}}(x))\eee^{\pi x} + x^{d+1} R(x).
$$
After some transformations, we arrive at
\begin{equation}\label{eq:G_even_odd}
G_{\text{even}}(x) (\eee^{\pi x} - 1)  + x^{d+1} R(x)  =  G_{\text{odd}}(x) (\eee^{\pi x} + 1).
\end{equation}

\vspace*{2mm}
\noindent
\textit{Case 1: $d$ is even.} From~\eqref{eq:g_ell_even_codim} we know that $g_\ell = A[d,d-\ell]$ provided $\ell\in \{0,\ldots,d\}$ is even. Therefore,
$$
x^d G_{\text{even}}\left(\frac 1x\right)
=
\sum_{\substack{\ell=0,2,4,\ldots\\ \ell\leq d}} g_\ell x^{d-\ell}
=
\sum_{\substack{k=0,2,4,\ldots\\ k\leq d}} A[d,k] x^k
=
Q_d(x),
$$
where we recall that $Q_d(x)$ is given by~\eqref{eq:def_Q_n}. From~\eqref{eq:G_even_odd} it follows that
$$
x^d G_{\text{odd}}\left(\frac 1x\right)
=
x^d G_{\text{even}}\left(\frac 1x\right) \cdot \frac{\eee^{\pi/x} - 1}{\eee^{\pi/x} + 1} + \frac 1x R\left(\frac 1x\right)
=
Q_d(x) \tanh \left(\frac{\pi}{2x}\right) +  \frac 1x R\left(\frac 1x\right).
$$
Note that $\frac 1x R(\frac 1x)$ does not contain any nonnegative powers of $x$. Comparing the coefficients, we arrive at
$$
g_\ell
=
[x^{d-\ell}] \left(x^d G_{\text{odd}}\left(\frac 1x\right)\right)
=
[x^{d-\ell}] \left(Q_d(x) \tanh \left(\frac{\pi}{2x}\right) \right)
$$
provided $\ell\in \{0,\ldots,d\}$ is odd. This proves the first case of~\eqref{eq:g_l_need}.

\vspace*{2mm}
\noindent
\textit{Case 2: $d$ is odd.} By~\eqref{eq:g_ell_even_codim}, we already know that $g_\ell = A[d,d-\ell]$ provided that $\ell\in \{0,\ldots,d\}$ is odd.  Hence,
$$
x^{d} G_{\text{odd}}\left(\frac 1x\right)
=
\sum_{\substack{\ell=1,3,5,\ldots\\ \ell\leq d}} g_\ell x^{d-\ell}
=
\sum_{\substack{k=0,2,4,\ldots\\ k\leq d}} A[d,k] x^k
=
Q_d(x).
$$
From~\eqref{eq:G_even_odd} it follows that
$$
x^d G_{\text{even}}\left(\frac 1x\right)
=
x^d G_{\text{odd}}\left(\frac 1x\right)\cdot  \frac{\eee^{\pi/x} + 1}{\eee^{\pi/x} - 1} - \frac 1x R\left(\frac 1x\right)
=
Q_d(x) \cotanh \left(\frac{\pi}{2x}\right) -  \frac 1x R\left(\frac 1x\right).
$$
Comparing the coefficients, we obtain
$$
g_\ell
=
[x^{d-\ell}] \left(x^d G_{\text{even}}\left(\frac 1x\right)\right)
=
[x^{d-\ell}] \left(Q_d(x) \cotanh \left(\frac{\pi}{2x}\right) \right)
$$
provided $\ell\in \{0,\ldots,d\}$ is even.  This proves the second case of~\eqref{eq:g_l_need} and completes the proof of Theorem~\ref{theo:main}.
\hfill $\Box$

\section{Further proofs}

\subsection{Expected internal angle sums: Removing parity restrictions}\label{subsec:removing_parity}
Now that we proved Theorem~\ref{theo:main} without parity restrictions, we are able to show that the formula for the quantities $\tilde \JJ_{n,k}(\frac n2)$ stated in Proposition~\ref{prop:internal_explicit} holds for all $k\in \{1,\ldots,n\}$ irrespective of the parity.
\begin{proposition}\label{prop:internal_explicit_no_parity}
For all $n\in \N$ and $k\in \{1,\ldots, n\}$ we have
\begin{align}
\tilde \JJ_{n,k}\left(\frac n2\right)
&=
\frac{\pi^{k-n}}{k!}\cdot \frac{n}{2 \tilde c_{1, \frac{n+1}{2}}} \cdot (A[n,k] - A[n-2,k]),
\label{eq:internal_explicit_no_parity1}\\
&=
\frac{\pi^{k-n}}{k!}\cdot \frac{n}{2 \tilde c_{1, \frac{n+1}{2}}} \cdot (n-1)^2 A[n-2,k-2],\label{eq:internal_explicit_no_parity2}
\end{align}
where the numbers $A[n,k]$, $n\in\N_0$, $k\in\Z$, are given by~\eqref{eq:def_A_n_k} and~\eqref{eq:def_Q_n}.
\end{proposition}
\begin{remark}\label{rem:convention2}
In the case $n=k=1$ we recall the conventions $0^2 A[-1,-1] = 2/\pi$ and $A[-1,1] = 0$.
\end{remark}
\begin{proof}[Proof of Proposition~\ref{prop:internal_explicit_no_parity}]
First of all, note that~\eqref{eq:internal_explicit_no_parity1} and~\eqref{eq:internal_explicit_no_parity2} are equivalent  by Lemma~\ref{lem:rec_first_kind} and, in the case $n=k=1$, by Remark~\ref{rem:convention2}.
The proof is essentially a repetition of the proof of Theorem~\ref{theo:main_even} (given after Proposition~\ref{prop:internal_explicit}) in reversed order. Using Theorem~\ref{theo:main} (with even or odd codimension!), recalling~\eqref{eq:E_f_ell_Z_d_Poisson_Process} and finally applying~\eqref{eq:lim_E_f_k_C_n}, we obtain
\begin{equation}\label{eq:tech1}
\frac{\pi^k}{k!} A[n,k] = \E f_{n-k}(\mathcal Z_{n}) = \E f_{k-1}(\conv \Pi_{n,1})  = \sum_{\substack{s=0,1,\ldots\\m:=n-2s\geq k}} \frac {2}{m} \pi^m \tilde c_{1, \frac{m+1}{2}} \tilde \JJ_{m,k}\left(\frac m2\right),
\end{equation}
for all $n\in\N$ and $k\in \{1,\ldots, n\}$. Assuming that $n\geq 3$ and replacing $n$ by $n-2$, we can write
\begin{equation}\label{eq:tech2}
\frac{\pi^k}{k!} A[n-2,k]  = \sum_{\substack{s=0,1,\ldots\\m:=n-2-2s\geq k}} \frac {2}{m} \pi^m \tilde c_{1, \frac{m+1}{2}} \tilde \JJ_{m,k}\left(\frac m2\right),
\end{equation}
for all $k\in \{1,\ldots, n-2\}$. Subtracting~\eqref{eq:tech1} and~\eqref{eq:tech2}, we arrive at the required identity~\eqref{eq:internal_explicit_no_parity1} for all $k\in \{1,\ldots,n-2\}$. The remaining cases $k=n$ and $k=n-1$ (which also cover $n=1,2$) can be verified directly since $\tilde \JJ_{n,n}(n/2)=1$ and $\tilde \JJ_{n,n-1}(n/2)=n/2$. The formulae for $A[n,n]$ and $A[n,n-1]$ can be found in Proposition~\ref{prop_A_n_k_algorithm}~(ii).
\end{proof}

\subsection{Proof of Theorem~\ref{theo:spherical_polytope_f_vector}}
As it was observed in~\cite{convex_hull_sphere}, the $f$-vector of $C_n\cap \bS^d_+$ has the same distribution as the $f$-vector of the beta' polytope $\tilde P_{n,d}^{(d+1)/2}$. Indeed, the intersection of the random cone $C_n$ with the hyperplane $\{x_{0}=1\}$ (which is the tangent hyperplane to the half-sphere $\bS^{d}_+$ at its north pole) has the same distribution as the random polytope $\tilde P_{n,d}^{(d+1)/2}$; see~\cite[Proposition~2.2]{convex_hull_sphere}.

For the beta' polytope $\tilde P_{n,d}^{(d+1)/2}$ it was shown in~\cite[Theorem~1.14]{beta_polytopes} that
\begin{equation}\label{eq:f_vect_beta_prime}
\E f_k(\tilde P_{n,d}^{(d+1)/2}) = 2 \sum_{\substack{s=0,1,\ldots \\ d-2s\geq k+1}} \tilde \II_{n,d-2s}(1) \tilde \JJ_{d-2s,k+1}\left(\frac d2 -s\right)
\end{equation}
for all $k\in\{0,\ldots,d-1\}$.
By Lemma~\ref{lem:I_n_k_B_n_k} and Proposition~\ref{prop:internal_explicit_no_parity},
\begin{align*}
\tilde \II_{n,d-2s}(1)
&=
\frac {n!}{d-2s}  \tilde c_{1, \frac{d-2s+1}{2}} \pi^{d-2s-n} B\{n, d-2s\},\\
\tilde \JJ_{d-2s,k+1}\left(\frac d2-s\right)
&=
\frac{\pi^{k+1-d+2s}}{(k+1)!}\cdot \frac{d-2s}{2 \tilde c_{1, \frac{d-2s+1}{2}}} \cdot (d-2s-1)^2 A[d-2s-2,k-1].
\end{align*}
Plugging these values into~\eqref{eq:f_vect_beta_prime}, and performing numerous cancellations, we arrive at
\begin{equation}\label{eq:f_vect_spherical_poly_repeat}
\E f_k(\tilde P_{n,d}^{(d+1)/2}) = \frac{n!\pi^{k+1-n}}{(k+1)!} \sum_{\substack{s=0,1,\ldots \\ d-2s\geq k+1}} B\{n, d-2s\}(d-2s-1)^2 A[d-2s-2,k-1],
\end{equation}
for all $k\in\{0,\ldots,d-1\}$.
This completes the proof of Theorem~\ref{theo:spherical_polytope_f_vector}.
\hfill
$\Box$

\subsection{General combinatorial identities}
Now we are going to state combinatorial identities complementing Lemma~\ref{lemma:identity_A_B_1}.
\begin{lemma}\label{lemma:identity_A_B_2}
For all $n\in\N$ and all $k\in \{1,2,\ldots,n-1\}$ we have
\begin{align}
\sum_{\substack{s=0,1,\ldots \\ n-2s\geq k}} B\{n,n-2s\}  (n-2s-1)^2 A[n-2s-2,k-2] &= \frac{\pi^{n-k}}{(n-k)!},\label{eq:identity_2A}\\
\sum_{\substack{s=0,1,\ldots \\ n-2s\geq k+1}} B\{n,n-2s-1\}  (n-2s-2)^2 A[n-2s-3,k-2] &= \frac{\pi^{n-k}}{(n-k)!}. \label{eq:identity_2B}
\end{align}
\end{lemma}
\begin{remark}
The identities~\eqref{eq:identity_2A}, respectively, \eqref{eq:identity_2B}, are not true for $k=n$ because then the sum on the left-hand side is equal to $2$, respectively, $0$, whereas the right-hand side equals $1$ in both cases.
\end{remark}
\begin{remark}
Adding and subtracting~\eqref{eq:identity_2A} and~\eqref{eq:identity_2B} we obtain the identities
\begin{align}
\sum_{m=k}^n B\{n,m\}  (m-1)^2 A[m-2,k-2] &= 2\cdot  \frac{\pi^{n-k}}{(n-k)!},\label{eq:identity_2C}\\
\sum_{m=k}^n (-1)^{n-m} B\{n,m\}  (m-1)^2 A[m-2,k-2] &= 2 \cdot \delta_{nk}, \label{eq:identity_2D}
\end{align}
for all $n\in\N$ and $k\in \{1,\ldots,n\}$, where $\delta_{nk}$ is the Kronecker delta function.
\end{remark}
\begin{proof}[Proof of Lemma~\ref{lemma:identity_A_B_2}]
Proposition~\ref{prop:relations} with $\beta=\frac n2$ states that
\begin{align}
&\sum_{\substack{s=0,1,\ldots\\ n-2s\geq k}} \tilde \II_{n,n-2s}(1) \tilde \JJ_{n-2s,k} \left(\frac {n-2s} 2\right) = \frac 12 \binom nk,\label{eq:relation_U_V_rep}\\
&\sum_{\substack{s=0,1,\ldots\\ n-2s-1\geq k}} \tilde \II_{n,n-2s-1}(1) \tilde \JJ_{n-2s-1,k} \left(\frac {n-1-2s}2\right) = \frac 12 \binom nk,\label{eq:relation_U_V_one_more_rep}
\end{align}
for all $n\in \N$ and $k\in \{1,\ldots,n-1\}$. By Lemma~\ref{lem:I_n_k_B_n_k} and Proposition~\ref{prop:internal_explicit_no_parity},
\begin{align*}
\tilde \II_{n,n-2s}(1)
&=
\frac {n!}{n-2s}  \tilde c_{1, \frac{n-2s+1}{2}} \pi^{-2s} B\{n, n-2s\},\\
\tilde \JJ_{n-2s,k}\left(\frac {n-2s}2\right)
&=
\frac{\pi^{k-n+2s}}{k!}\cdot \frac{n-2s}{2 \tilde c_{1, \frac{n-2s+1}{2}}} \cdot (n-2s-1)^2 A[n-2s-2,k-2].
\end{align*}
Plugging these values into~\eqref{eq:relation_U_V_rep} and performing cancellations, we arrive at~\eqref{eq:identity_2A}. Starting with~\eqref{eq:relation_U_V_one_more_rep} and arguing in a similar way yields~\eqref{eq:identity_2B}.
\end{proof}


\subsection{Proof of Proposition~\ref{prop:f_vect_of_d+2_points}}
Let us prove the first identity. By Theorem~\ref{theo:spherical_polytope_f_vector} with $n=d+2$, for all $k\in \{0,\ldots,d-1\}$ we have
$$
\E f_{k} (C_{d+2}\cap \bS^d_+) = \frac{(d+2)!\pi^{k-d-1}}{(k+1)!} \sum_{\substack{s=0,1,\ldots \\ d-2s\geq k+1}} B\{d+2, d-2s\}(d-2s-1)^2 A[d-2s-2,k-1].
$$
Lemma~\ref{lemma:identity_A_B_2} with $n=d+2$ and $k$ replaced by $k+1$ states that
$$
\sum_{\substack{s=0,1,\ldots \\ d+2-2s\geq k+1}} B\{d+2,d+2-2s\}  (d+2-2s-1)^2 A[d-2s,k-1] = \frac{\pi^{d+1-k}}{(d+1-k)!}.
$$
Since the sums in the above two equations differ by just one term, we can write
\begin{align*}
\E f_{k} (C_{d+2}\cap \bS^d_+)
&=
\frac{(d+2)!\pi^{k-d-1}}{(k+1)!}
\left(\frac{\pi^{d+1-k}}{(d+1-k)!} - B\{d+2,d+2\}(d+1)^2 A[d,k-1]\right)\\
&=
\binom{d+2}{k+1} - \frac{(d+2)!\pi^{k-d-1}}{(k+1)!}B\{d+2,d+2\}(d+1)^2 A[d,k-1]\\
&=
\binom{d+2}{k+1} - \frac{(d+2)\pi^{k-d-1}}{(k+1)!} \frac{\sqrt \pi\, \Gamma\left(\frac {d+2}2\right)}{\Gamma\left(\frac{d+3}{2}\right)}  (d+1)^2 A[d,k-1]
\end{align*}
upon using the formula
$$
B\{d+2,d+2\} = \frac {1}{(d+1)!} \int_0^\pi (\sin x)^{d+1} \dint x = \frac {1}{(d+1)!} \frac{\sqrt \pi\, \Gamma\left(\frac {d+2}2\right)}{\Gamma\left(\frac{d+3}{2}\right)}.
$$
The proof of the first identity is complete.

The proof of the second identity is similar.  By Theorem~\ref{theo:spherical_polytope_f_vector} with $n=d+3$, for all $k\in \{0,\ldots,d-1\}$ we have
$$
\E f_{k} (C_{d+3}\cap \bS^d_+) = \frac{(d+3)!\pi^{k-d-2}}{(k+1)!} \sum_{\substack{s=0,1,\ldots \\ d-2s\geq k+1}} B\{d+3, d-2s\}(d-2s-1)^2 A[d-2s-2,k-1].
$$
Lemma~\ref{lemma:identity_A_B_2} with $n=d+3$ and $k$ replaced by  $k+1$ states that
$$
\sum_{\substack{s=0,1,\ldots \\ d+3-2s\geq k+2}} B\{d+3,d+2-2s\}  (d+1-2s)^2 A[d-2s,k-1] = \frac{\pi^{d+2-k}}{(d+2-k)!}.
$$
Again, the sums in the above two equations differ by just one term, so that we can write
\begin{align*}
\E f_{k} (C_{d+3}\cap \bS^d_+)
&=
\frac{(d+3)!\pi^{k-d-2}}{(k+1)!}
\left(\frac{\pi^{d+2-k}}{(d+2-k)!} - B\{d+3,d+2\}(d+1)^2 A[d,k-1]\right)\\
&=
\binom{d+3}{k+1} - \frac{(d+3)!\pi^{k-d-2}}{(k+1)!} B\{d+3,d+2\} (d+1)^2 A[d,k-1]\\
&=
\binom{d+3}{k+1} - \frac{(d+3) \pi^{k-d-1}}{(k+1)!} \frac{\sqrt \pi\, \Gamma\left(\frac {d+4}2\right)}{\Gamma\left(\frac{d+3}{2}\right)}  (d+1)^2 A[d,k-1]
\end{align*}
upon using the formula
$$
B\{d+3,d+2\}
=
\frac {1}{(d+1)!} \int_0^\pi (\sin x)^{d+1} x \dint x
=
\frac {1}{(d+1)!} \frac{\pi^{3/2}\, \Gamma\left(\frac {d+2}2\right)}{2\Gamma\left(\frac{d+3}{2}\right)}.
$$
This completes the proof of the second identity.
\hfill $\Box$

\subsection{Proof of Theorem~\ref{theo:expected_angle}}
Recall that $C_{n+1}$ is defined as the positive hull of the points $U_1,\ldots, U_{n+1}$ that are independent and uniformly distributed on the half-sphere $\bS_{+}^d$. For every $k\in \{1,\ldots,n+1\}$, the point $U_k$ is not a vertex of $C_{n+1}\cap \bS^d_+$ if and only if it is contained in the cone generated by the remaining points $U_i$, $i\in \{1,\ldots,n+1\}\backslash \{k\}$. Hence,
\begin{equation}\label{eq:efron}
(n+1) - \E f_0(C_{n+1}\cap \bS_+^d) =  (n+1) \P[U_{n+1} \in \pos (U_1,\ldots,U_n)] =  2 (n+1) \E \alpha (C_{n}).
\end{equation}
This spherical Efron-type identity was obtained in~\cite[Equation~(26)]{barany_etal} and is a special case of the more general identity proved in~\cite[Theorem 2.7]{convex_hull_sphere}.
It follows from~\eqref{eq:efron} that
\begin{align}
\E \alpha(C_n)
&=
\frac 12 \left(1 - \frac 1 {n+1} \E f_0(C_{n+1}\cap \bS_+^d) \right)\notag \\
&=
\frac 12 \Big(1 - \frac{n!}{\pi^{n}} \sum_{\substack{s=0,1,\ldots \\ d-2s\geq 1}} B\{n+1, d-2s\}
(d-2s-1)^2 A[d-2s-2,-1]\Big), \label{eq:alpha_C_n_proof}
\end{align}
where in the second equality we applied Theorem~\ref{theo:spherical_polytope_f_vector} with $k=0$.  Lemma~\ref{lemma:identity_A_B_2} with $n$ replaced by $n+1$ and $k=1$ implies that
\begin{multline*}
\frac{n!}{\pi^n} \sum_{\substack{m\in \{1,\ldots,n+1\} \\ m\not \equiv n \Mod{2}}} B\{n+1,m\}(m-1)^2 A[m-2,-1] \\
=
\frac{n!}{\pi^n}\sum_{\substack{m\in \{1,\ldots,n+1\} \\ m \equiv n \Mod{2}}} B\{n+1,m\} (m-1)^2 A[m-2,-1]
=
1.
\end{multline*}
Replacing the term $1$ in~\eqref{eq:alpha_C_n_proof} by one of the above sums depending on the parity of $d$,  we can write~\eqref{eq:alpha_C_n_proof} in the form
$$
\E \alpha(C_n) = \frac{n!}{2\pi^{n}} \sum_{\substack{m\in \{d+2,\ldots,n+1\} \\ m \equiv d \Mod{2}}} B\{n+1, m\}
(m-1)^2 A[m-2,-1],
$$
which completes the proof.
\hfill $\Box$

\section*{Acknowledgements}
Supported by the German Research Foundation under Germany's Excellence Strategy  EXC 2044 -- 390685587, Mathematics M\"unster: Dynamics - Geometry - Structure. The author is grateful to Christoph Th\"ale for useful comments, as well as to an unknown referee for an extremely thorough handling of the manuscript and numerous useful suggestions.


\bibliography{poisson_zero_bib}
\bibliographystyle{plainnat}
\addresseshere

\newpage

\appendix
\renewcommand{\sectionname}{}

\section{Tables}\label{sec:table}

\begin{table}[h!]
\begin{center}
\begin{tabular}{||c||l||}
\hhline{|==|}
\parbox[0pt][3em][c]{0cm}{} $d$ & $\E f_0(\mathcal Z_{d}),\; \E f_1(\mathcal Z_{d}),\; \ldots,\; \E f_{d-1}(\mathcal Z_{d})$\\
\hhline{||==||}
\parbox[0pt][3em][c]{0cm}{}
$1$ & $2$\\
\hhline{||--||}
\parbox[0pt][3em][c]{0cm}{} $2$ & $\frac{\pi ^2}{2},\frac{\pi ^2}{2}$\\
\hhline{||--||}
\parbox[0pt][3em][c]{0cm}{} $3$ & $\frac{4 \pi ^2}{3},2 \pi ^2,2 \left(1+\frac{\pi ^2}{3}\right)$\\
\hhline{||--||}
\parbox[0pt][3em][c]{0cm}{} $4$ & $\frac{3 \pi ^4}{8},\frac{3 \pi ^4}{4},5 \pi ^2,5 \pi ^2-\frac{3 \pi ^4}{8}$ \\
\hhline{||--||}
\parbox[0pt][3em][c]{0cm}{} $5$ & $\frac{16 \pi ^4}{15},\frac{8 \pi ^4}{3},\frac{4}{9} \pi ^2 \left(15+4 \pi ^2\right),10 \pi ^2,2+\frac{10 \pi ^2}{3}-\frac{8 \pi ^4}{45}$\\
\hhline{||--||}
\parbox[0pt][3em][c]{0cm}{} $6$ & $\frac{5 \pi ^6}{16},\frac{15 \pi ^6}{16},\frac{259 \pi ^4}{24},\frac{1}{48} \left(1036 \pi ^4-75 \pi ^6\right),\frac{35 \pi ^2}{2},\frac{1}{48} \pi ^2 \left(840-518 \pi ^2+45 \pi ^4\right)$ \\
\hhline{||--||}
\parbox[0pt][4em][c]{0cm}{} $7$ & $\frac{32 \pi ^6}{35},\frac{16 \pi ^6}{5},\frac{4}{15} \pi ^4 \left(49+12 \pi ^2\right),\frac{98 \pi ^4}{3},\frac{4}{45} \pi ^2 \left(210+245 \pi ^2-12 \pi ^4\right),28 \pi ^2,2+\frac{28 \pi ^2}{3}-\frac{98 \pi ^4}{45}+\frac{16 \pi ^6}{105}$\\
\hhline{||--||}
\parbox[0pt][4em][c]{0cm}{} $8$ & \makecell[l]{$\frac{35 \pi ^8}{128},\frac{35 \pi ^8}{32},\frac{3229 \pi ^6}{180},\frac{3229 \pi ^6}{60}-\frac{245 \pi ^8}{64},\frac{329 \pi ^4}{4},\frac{329 \pi ^4}{2}-\frac{3229 \pi ^6}{36}+\frac{245 \pi ^8}{32},42 \pi ^2,$
\\
$42 \pi ^2-\frac{329 \pi ^4}{4}+\frac{3229 \pi ^6}{60}-\frac{595 \pi ^8}{128}$} \\
\hhline{||--||}
\parbox[0pt][4em][c]{0cm}{} $9$ & \makecell[l]{$\frac{256 \pi ^8}{315},\frac{128 \pi ^8}{35},\frac{32}{315} \pi ^6 \left(205+48 \pi ^2\right),\frac{656 \pi ^6}{9},\frac{364 \pi ^4}{5}+\frac{656 \pi ^6}{9}-\frac{256 \pi ^8}{75},182 \pi ^4,$\\
$40 \pi ^2+\frac{364 \pi ^4}{3}-\frac{656 \pi ^6}{27}+\frac{512 \pi ^8}{315},60 \pi ^2,2+20 \pi ^2-\frac{182 \pi ^4}{15}+\frac{656 \pi ^6}{189}-\frac{128 \pi ^8}{525}$}\\
\hhline{||--||}
\parbox[0pt][4em][c]{0cm}{}
$10$ & \makecell[l]{$\frac{63 \pi ^{10}}{256},\frac{315 \pi ^{10}}{256},\frac{117469 \pi ^8}{4480},\frac{\pi ^8 \left(469876-33075 \pi ^2\right)}{4480},\frac{17281 \pi ^6}{72},\frac{\pi ^6 \left(1382480-704814 \pi ^2+59535 \pi ^4\right)}{1920},\frac{1463 \pi ^4}{4},$\\
$\frac{1463 \pi ^4}{2}-\frac{86405 \pi ^6}{72}+\frac{117469 \pi ^8}{160}-\frac{16065 \pi ^{10}}{256},\frac{165 \pi ^2}{2},\frac{165 \pi ^2}{2}-\frac{1463 \pi ^4}{4}+\frac{17281 \pi ^6}{24}-\frac{1996973 \pi ^8}{4480}+\frac{9765 \pi ^{10}}{256}$}\\
\hhline{|==|}
\end{tabular}
\caption{Expected $f$-vector of the Poisson zero polytope in dimensions $d\in \{1,\ldots,10$\}}
\label{tab:f_vect_Poisson}
\end{center}
\end{table}

\newpage
\begin{table}[h]
\begin{center}
\begin{tabular}{||c||c|c|c|c|c|c|c|c||}
\hhline{|=========|}
\parbox[0pt][3em][c]{0cm}{}
$A[n,k]$& $k=0$ & $k=2$ & $k=4$ & $k=6$ & $k=8$ & $k=10$ & $k=12$ & $k=14$\\
\hhline{||=|=|=|=|=|=|=|=|=||}
\parbox[0pt][3em][c]{0cm}{}
 $n=1$ & $1$ & $0$ & $0$ & $0$ & $0$ & $0$ & $0$ & $0$ \\
\hhline{||---------||}
\parbox[0pt][3em][c]{0cm}{}
 $n=2$ & $1$ & $1$ & $0$ & $0$ & $0$ & $0$ & $0$ & $0$ \\
\hhline{||---------||}
\parbox[0pt][3em][c]{0cm}{}
 $n=3$ & $1$ & $4$ & $0$ & $0$ & $0$ & $0$ & $0$ & $0$ \\
 \hhline{||---------||}
\parbox[0pt][3em][c]{0cm}{}
 $n=4$ & $1$ & $10$ & $9$ & $0$ & $0$ & $0$ & $0$ & $0$ \\
 \hhline{||---------||}
\parbox[0pt][3em][c]{0cm}{}
 $n=5$ & $1$ & $20$ & $64$ & $0$ & $0$ & $0$ & $0$ & $0$ \\
 \hhline{||---------||}
\parbox[0pt][3em][c]{0cm}{}
 $n=6$ & $1$ & $35$ & $259$ & $225$ & $0$ & $0$ & $0$ & $0$ \\
 \hhline{||---------||}
\parbox[0pt][3em][c]{0cm}{}
 $n=7$ & $1$ & $56$ & $784$ & $2304$ & $0$ & $0$ & $0$ & $0$ \\
 \hhline{||---------||}
\parbox[0pt][3em][c]{0cm}{}
 $n=8$ & $1$ & $84$ & $1974$ & $12916$ & $11025$ & $0$ & $0$ & $0$ \\
 \hhline{||---------||}
\parbox[0pt][3em][c]{0cm}{}
 $n=9$ & $1$ & $120$ & $4368$ & $52480$ & $147456$ & $0$ & $0$ & $0$ \\
 \hhline{||---------||}
\parbox[0pt][3em][c]{0cm}{}
 $n=10$& $1$ & $165$ & $8778$ & $172810$ & $1057221$ & $893025$ & $0$ & $0$ \\
 \hhline{||---------||}
\parbox[0pt][3em][c]{0cm}{}
 $n=11$& $1$ & $220$ & $16368$ & $489280$ & $5395456$ & $14745600$ & $0$ & $0$ \\
 \hhline{||---------||}
\parbox[0pt][3em][c]{0cm}{}
 $n=12$& $1$ & $286$ & $28743$ & $1234948$ & $21967231$ & $128816766$ & $108056025$ & $0$ \\
 \hhline{||---------||}
\parbox[0pt][3em][c]{0cm}{}
 $n=13$& $1$ & $364$ & $48048$ & $2846272$ & $75851776$ & $791691264$ & $2123366400$ & $0$ \\
 \hhline{||---------||}
\parbox[0pt][3em][c]{0cm}{}
 $n=14$& $1$ & $455$ & $77077$ & $6092515$ & $230673443$ & $3841278805$ & $21878089479$ & $18261468225$ \\
\hhline{|=========|}
\end{tabular}
\caption{The values of $A[n,k]$ for $n\in \{1,\ldots, 14\}$ and even $k\in \{0,2,4,\ldots, 14\}$.}
\label{tab:A_n_k_0}
\end{center}
\end{table}

\newpage

\begin{table}[h]
\begin{center}
\begin{tabular}{||c||c|c|c|c|c|c||}
\hhline{|=======|}
\parbox[0pt][3em][c]{0cm}{}
$A[n,k]$& $k=0$ & $k=1$ & $k=2$ & $k=3$ & $k=4$ & $k=5$ \\
\hhline{||==|=|=|=|=|=||}
\parbox[0pt][3em][c]{0cm}{}
$n=0$ &   $1$ & $0$ & $0$ & $0$ & $0$ & $0$ \\
\hhline{||--|-|-|-|-|-||}
\parbox[0pt][3em][c]{0cm}{}
$n=1$ &   $1$ & $\frac{2}{\pi }$ & $0$ & $0$ & $0$ & $0$ \\
\hhline{||--|-|-|-|-|-||}
\parbox[0pt][3em][c]{0cm}{}
$n=2$ & $1$ & $\frac{\pi }{2}$ & $1$ & $0$ & $0$ & $0$ \\
\hhline{||--|-|-|-|-|-||}
\parbox[0pt][3em][c]{0cm}{}
$n=3$ &  $1$ & $\frac{2}{\pi }+\frac{2 \pi }{3}$ & $4$ & $\frac{8}{\pi }$ & $0$ & $0$  \\
\hhline{||--|-|-|-|-|-||}
\parbox[0pt][3em][c]{0cm}{}
$n=4$ &  $1$ & $5 \pi -\frac{3 \pi ^3}{8}$ & $10$ & $\frac{9 \pi }{2}$ & $9$ & $0$  \\
\hhline{||--|-|-|-|-|-||}
\parbox[0pt][3em][c]{0cm}{}
$n=5$ &  $1$ & $\frac{2}{\pi }+\frac{10 \pi }{3}-\frac{8 \pi ^3}{45}$ & $20$ & $\frac{40}{\pi }+\frac{32 \pi }{3}$ & $64$ & $\frac{128}{\pi }$ \\
\hhline{||--|-|-|-|-|-||}
\parbox[0pt][3em][c]{0cm}{}
$n=6$ &   $1$ & $\frac{1}{48} \pi  \left(840-518 \pi ^2+45 \pi ^4\right)$ & $35$ & $\frac{259 \pi }{2}-\frac{75 \pi ^3}{8}$ & $259$ & $\frac{225 \pi }{2}$\\
\hhline{||--|-|-|-|-|-||}
\parbox[0pt][3em][c]{0cm}{}
$n=7$ &  $1$ & $\frac{2}{\pi }+\frac{28 \pi }{3}-\frac{98 \pi ^3}{45}+\frac{16 \pi ^5}{105}$ & $56$ & $\frac{112}{\pi }+\frac{392 \pi }{3}-\frac{32 \pi ^3}{5}$ & $784$ & $\frac{1568}{\pi }+384 \pi$\\
\hhline{||--|-|-|-|-|-||}
\parbox[0pt][3em][c]{0cm}{}
$n=8$ &  $1$ & $42 \pi -\frac{329 \pi ^3}{4}+\frac{3229 \pi ^5}{60}-\frac{595 \pi ^7}{128}$ & $84$ & $987 \pi -\frac{3229 \pi ^3}{6}+\frac{735 \pi ^5}{16}$ & $1974$ & $6458 \pi -\frac{3675 \pi ^3}{8}$ \\
\hhline{|=======|}
\end{tabular}
\caption{The values of $A[n,k]$ for $n\in \{0,\ldots, 8\}$ and $k\in \{0,\ldots,5\}$.}
\label{tab:A_n_k}
\end{center}
\end{table}

\newpage

\begin{table}[h]
\begin{center}
\begin{tabular}{||c||c|c|c|c||}
\hhline{|=====|}
\parbox[0pt][3em][c]{0cm}{}
$B\{n,k\}$&$k=1$ & $k=2$ & $k=3$ & $k=4$\\
\hhline{||=||=|=|=|=||}
\parbox[0pt][3em][c]{0cm}{}
$n=1$ & $\pi$  & $0$ & $0$ & $0$ \\
\hhline{||-----||}
\parbox[0pt][3em][c]{0cm}{}
$n=2$ & $\frac{\pi ^2}{2}$ & $2$ & $0$ & $0$ \\
\hhline{||-----||}
\parbox[0pt][3em][c]{0cm}{}
$n=3$ & $\frac{\pi ^3}{6}$ & $\pi$  & $\frac{\pi }{4}$ & $0$ \\
\hhline{||-----||}
\parbox[0pt][3em][c]{0cm}{}
$n=4$ & $\frac{\pi ^4}{24}$ & $\frac{1}{2} \left(-4+\pi ^2\right)$ & $\frac{\pi ^2}{8}$ & $\frac{2}{9}$ \\
\hhline{||-----||}
\parbox[0pt][3em][c]{0cm}{}
$n=5$ & $\frac{\pi ^5}{120}$ & $\frac{1}{6} \pi  \left(-6+\pi ^2\right)$ & $\frac{1}{48} \pi  \left(-3+2 \pi ^2\right)$ & $\frac{\pi }{9}$ \\
\hhline{||-----||}
\parbox[0pt][3em][c]{0cm}{}
$n=6$ & $\frac{\pi ^6}{720}$ & $\frac{1}{24} \left(48-12 \pi ^2+\pi ^4\right)$ & $\frac{1}{96} \pi ^2 \left(-3+\pi ^2\right)$ & $\frac{1}{162} \left(-40+9 \pi ^2\right)$ \\
\hhline{||-----||}
\parbox[0pt][3em][c]{0cm}{}
$n=7$ & $\frac{\pi ^7}{5040}$ & $\pi -\frac{\pi ^3}{6}+\frac{\pi ^5}{120}$ & $\frac{1}{960} \pi  \left(15-10 \pi ^2+2 \pi ^4\right)$ & $\frac{1}{162} \pi  \left(-20+3 \pi ^2\right)$ \\
\hhline{||-----||}
\parbox[0pt][3em][c]{0cm}{}
$n=8$ &$\frac{\pi ^8}{40320}$ & $-2+\frac{\pi ^2}{2}-\frac{\pi ^4}{24}+\frac{\pi ^6}{720}$ & $\frac{\pi ^2 \left(45-15 \pi ^2+2 \pi ^4\right)}{5760}$ & $\frac{1456-360 \pi ^2+27 \pi ^4}{5832}$ \\
\hhline{||-----||}
\parbox[0pt][3em][c]{0cm}{}
$n=9$ & $\frac{\pi ^9}{362880}$ & $\frac{\pi  \left(-5040+840 \pi ^2-42 \pi ^4+\pi ^6\right)}{5040}$ & $\frac{\pi  \left(-315+210 \pi ^2-42 \pi ^4+4 \pi ^6\right)}{80640}$ & $\frac{\pi  \left(3640-600 \pi ^2+27 \pi ^4\right)}{29160}$ \\
\hhline{||-----||}
\parbox[0pt][3em][c]{0cm}{}
$n=10$ & $\frac{\pi ^{10}}{3628800}$ & $2-\frac{\pi ^2}{2}+\frac{\pi ^4}{24}-\frac{\pi ^6}{720}+\frac{\pi ^8}{40320}$ & $\frac{\pi ^2 \left(-315+105 \pi ^2-14 \pi ^4+\pi ^6\right)}{161280}$ & $\frac{-131200+32760 \pi ^2-2700 \pi ^4+81 \pi ^6}{524880}$ \\
\hhline{|=====|}
\end{tabular}
\caption{The values of $B\{n,k\}$ for $n\in \{1,\ldots, 10\}$ and $k\in \{1,\ldots,4\}$.}
\label{tab:B_n_k}
\end{center}
\end{table}

\newpage

\begin{table}[h]
\begin{center}
\begin{tabular}{||c||c||}
\hhline{|==|}
\parbox[0pt][3em][c]{0cm}{} $d$ & $P(d)$\\
\hhline{||==||}
\parbox[0pt][3em][c]{0cm}{} $1$ & $1$\\
\hhline{||--||}
\parbox[0pt][3em][c]{0cm}{} $2$ & $-2 + \frac{24}{\pi^2}$\\
\hhline{||--||}
\parbox[0pt][3em][c]{0cm}{} $3$ & $\frac{5}{\pi^2}-\frac{1}{3}$\\
\hhline{||--||}
\parbox[0pt][3em][c]{0cm}{} $4$ & $\frac{80}{\pi^4}+6-\frac{200}{3 \pi^2}$\\
\hhline{||--||}
\parbox[0pt][3em][c]{0cm}{} $5$ & $\frac{1}{3}+\frac{105}{8 \pi ^4}-\frac{35}{8 \pi ^2}$\\
\hhline{||--||}
\parbox[0pt][3em][c]{0cm}{} $6$ & $-\frac{1568}{3 \pi ^4}+\frac{896}{5 \pi ^6}-34+\frac{29008}{75 \pi ^2}$ \\
\hhline{||--||}
\parbox[0pt][4em][c]{0cm}{} $7$ & $-\frac{3}{5}-\frac{49}{2 \pi ^4}+\frac{105}{4 \pi ^6}+\frac{49}{6 \pi ^2}$\\
\hhline{||--||}
\parbox[0pt][4em][c]{0cm}{} $8$ & $\frac{27072}{5 \pi ^4}-\frac{2304}{\pi ^6}+\frac{2304}{7 \pi ^8}+310-\frac{878288}{245 \pi ^2}$ \\
\hhline{||--||}
\parbox[0pt][4em][c]{0cm}{} $9$ &  $\frac{5 \left(3465-6930 \pi ^2+6006 \pi ^4-1804 \pi ^6+128 \pi ^8\right)}{384 \pi ^8}$\\
\hhline{||--||}
\parbox[0pt][4em][c]{0cm}{} $10$ &
$\frac{3548160-48787200 \pi ^2+259547904 \pi ^4-517047520 \pi ^6+320455432 \pi ^8-27425790 \pi ^{10}}{6615 \pi ^{10}}$
\\
\hhline{|==|}
\end{tabular}
\caption{The first few values of the probability $P(d)$ defined in Section~\ref{subsec:sylvester}.}
\label{tab:sylvester}
\end{center}
\end{table}

\end{document}